\documentclass[12pt]{amsart}
\usepackage{SSdefn}

\usepackage{eucal}
\usepackage{comment}
\usepackage{breakurl}
\usepackage{tikz}

\DeclareMathOperator{\Split}{Split}
\DeclareMathOperator{\SSplit}{SSplit}
\DeclareMathOperator{\DSplit}{DSplit}

\DeclareMathOperator{\SFact}{SFact}
\DeclareMathOperator{\DFact}{DFact}
\DeclareMathOperator{\Ext}{Ext}

\newcommand{\Fl}{\mathbf{Fl}}

\newcommand{\gl}{\mathfrak{gl}}
\newcommand{\ord}{\mathrm{ord}}

\newcommand{\sing}{\mathrm{sing}}

\newcommand{\univ}{\mathrm{univ}}

\newcommand{\stacks}[1]{\cite[Tag \href{https://stacks.math.columbia.edu/tag/#1}{#1}]{stacks-project}}

\title[Flag supervarieties and determinantal ideals II]{Cohomology of flag supervarieties and \\ resolutions of determinantal ideals. II}
\date{December 30, 2024}

\author{Steven V Sam}
\address{Department of Mathematics, University of California, San Diego, CA}
\email{\href{mailto:ssam@ucsd.edu}{ssam@ucsd.edu}}
\urladdr{\url{http://math.ucsd.edu/~ssam/}}
\thanks{SS was supported by NSF grant DMS-2302149.}

\author{Andrew Snowden}
\address{Department of Mathematics, University of Michigan, Ann Arbor, MI}
\email{\href{mailto:asnowden@umich.edu}{asnowden@umich.edu}}
\urladdr{\url{http://www-personal.umich.edu/~asnowden/}}
\thanks{AS was supported by NSF grant DMS-2301871.}

\begin{document}

\begin{abstract}
  We compute the coherent cohomology of the structure sheaf of complex periplectic Grassmannians. In particular, we show that it can be decomposed as a tensor product of the singular cohomology ring of a Grassmannian for either the symplectic or orthogonal group together with a semisimple representation of the periplectic Lie supergroup. The restriction of the latter to its even subgroup has an explicit multiplicity-free description in terms of Schur functors and is closely related to syzygies of (skew-)symmetric determinantal ideals. We develop tools for studying splitting rings for Coxeter groups of types BC and D, which may be of independent interest.
\end{abstract}

\maketitle
\tableofcontents

\section{Introduction}

This article is a continuation of \cite{superres} in which we studied the close connection between the coherent cohomology of the structure sheaf of (complex) super Grassmannians and the syzygies of determinantal varieties in the space of generic matrices. In this article, we consider a parallel situation between periplectic Grassmannians and syzygies of determinantal varieties in the spaces of symmetric and skew-symmetric matrices.

The general problem of computing the cohomology of homogeneous bundles (or even just line bundles) on  homogeneous supervarieties, such as the super Grassmannian, has been considered for quite some time now (for a sample of literature, see \cite{coulembier, gruson, penkov, penkov-serganova}) and is largely open. Its classical (non-super) counterpart has been understood for a long time now and is solved by the Borel--Weil--Bott theorem. In this case, the cohomology groups are zero except for possibly one degree, and is an irreducible representation when this happens. Both of these properties fail for general line bundles on homogeneous supervarieties, but they are well-behaved for ``typical'' weights. The structure sheaf, while fundamentally important, is unfortunately not a line bundle with a typical weight. As we will see, the two properties fail but the cohomology groups will still be highly structured.

The general setup is as follows: given a complex supervariety $X$, let $\cJ \subset \cO_X$ denote the ideal sheaf of $\cO_{X_\ord}$, where $X_{\ord}$ is the underlying variety, and let $\gr \cO_X = \bigoplus_{n \ge 0} \cJ^n/\cJ^{n+1}$ be the associated graded $\cO_{X_\ord}$-algebra of $\cO_X$ with respect to the $\cJ$-adic filtration. Then $\gr \cO_X$ is canonically a quotient of the exterior algebra $\bigwedge^\bullet(\cJ/\cJ^2)$, and we call $X$ smooth if this quotient map is an isomorphism and $X_\ord$ is a smooth variety.

Hence, when $X$ is smooth, we get a spectral sequence for computing the cohomology of $\cO_X$ whose input is the cohomology of exterior powers $\bigwedge^k(\cJ/\cJ^2)$. Our examples of interest, super Grassmannians and periplectic Grassmannians, are smooth. Furthermore, these examples are homogeneous spaces for a supergroup $\bG$, and the vector bundle $\cJ/\cJ^2$ is naturally a subbundle of a trivial bundle $\fg_1 \otimes \cO_{X_{\ord}}$; here $\fg$ is the Lie superalgebra of $\bG$, and $\fg_1$ is its odd subspace. Let $\eta$ denote the quotient bundle and $Y = \Spec_{\cO_{X_\ord}}(\Sym(\eta))$. The projection $\pi \colon Y \to \fg_1$ is an example of a Kempf collapsing, and in this situation, the cohomology of $\bigwedge^k(\cJ/\cJ^2)$ takes on another interpretation as computing the (hyper) Tor groups of the derived pushforward of $\cO_Y$:
\[
  \Tor_i^{\cO_{\fg_1}}(\rR^\bullet \pi_* \cO_Y, \bC)_{i+j} = \rH^j(X_{\ord}, \bigwedge^i(\cJ/\cJ^2)).
\]
This general scenario is studied in great detail in the book \cite{weyman} and was first used to study determinantal varieties in \cite{lascoux}. The higher direct images of $\cO_Y$ are 0 for super Grassmannians (as studied in \cite{superres}) and the periplectic Grassmannians (as we will show), so that the left hand side above is computing the usual Tor groups (or syzygies) of the algebraic variety $\Spec(\pi_* \cO_Y)$. For the super Grassmannian, $\fg_1$ is the space of pairs $(f,g)$ where $f$ is an $m \times n$ matrix and $g$ is an $n \times m$ matrix.

What initially grabbed our interest is that, in this case, the support of $\pi_* \cO_Y$ is a determinantal variety, i.e., defined by a rank condition on $f$ or $g$. Furthermore, $\pi_*\cO_Y$ turns out to be a free module over the structure sheaf of its support, so that the Tor groups above are very easy to relate to the Tor groups of the determinantal variety. As an added bonus, we can use what is known about determinantal varieties to show that the spectral sequence which computes $\rH^*(X,\cO_X)$ is degenerate, and hence we can extract quite a lot of information.

In particular, all of these facts taken together imply that the direct sum of the Tor groups carry an action of the supergroup $\bG$ (because it naturally acts on $\rH^*(X,\cO_X)$). It had been known before that this action exists via subtle algebraic arguments (for example, see \cite{pragacz-weyman, akin-weyman, akin-weyman3, raicu-weyman}), so the novelty here is to give a geometric explanation of this mysterious action.

The case under consideration in this article, the periplectic Grassmannian, shares many of the pleasant features of the previous case. The main change is that $\fg_1$ is now the space of pairs $(f,g)$ where $f$ is a skew-symmetric $n \times n$ matrix and $g$ is a symmetric $n \times n$ matrix, but the support of $\pi_*\cO_Y$ is still a variety defined by a rank condition on either $f$ or $g$. In the first case $\pi_*\cO_Y$ is a free module over the structure sheaf of its support, but in the second case, one actually needs to consider an auxiliary double cover of its support.

Nonetheless, the upshot is that a similar conclusion holds: the direct sum of the Tor groups has an action of the periplectic group $\bG$, which again was previously shown to exist by subtle algebraic arguments in \cite{sam}.

The bulk of the technical work in this article is to prove the freeness statement of $\pi_*\cO_Y$ and to prove that the higher direct images vanish. As a bonus, we are able to identify the ring structure of $\pi_*\cO_Y$ explicitly. For the super Grassmannian, $\pi_*\cO_Y$ is constructed as a splitting ring over the coordinate ring of its support. This is a certain ring-theoretic analogue of the splitting field of a polynomial whose Galois group is the full symmetric group. Up to a technical modification, the polynomial in question is the characteristic polynomial of the product $fg$. A great deal of the work in \cite{superres} to identify this ring structure involved showing that this ring is normal (integrally closed).

We can attempt to do something similar in our current setup, but the splitting ring will no longer be normal. Roughly speaking, the characteristic polynomial is now an even polynomial, which relates to the fact that the Galois group is no longer the full symmetric group, but instead the wreath product of the symmetric group with $\bZ/2$ (i.e., the type B Weyl group). This leads us to the notion of a ``signed splitting ring'' and a certain variant which we call the ``type D splitting ring'' which corresponds to the type D Weyl group. The role of the latter is surprisingly subtle, but ultimately its use is dictated by the fact that Pfaffians exist, i.e., that the determinant of a skew-symmetric matrix whose entries are indeterminates is actually the square of another polynomial. Nonetheless, it is a pleasant feature that all 3 infinite families of Weyl groups play parallel roles in these computations.

Hence, in broad strokes, while the outline of this paper is similar to that of \cite{superres}, many technical aspects end up being more subtle here.

Here is a brief outline of the contents. \S\ref{sec:typeB} is devoted to basic ring-theoretic aspects of the signed splitting rings and \S\ref{sec:typeD} considers the type D analogue. We pay particular attention to developing results for proving that these rings are normal in terms of the discriminant of the given  polynomial. Unfortunately, neither one generalizes the other, so we develop them separately for the sake of clarity. In \S\ref{s:pedetvar}, we apply the splitting ring constructions from the previous sections to determinantal varieties in the symmetric and skew-symmetric settings. One subtlety here is to incorporate a certain double cover of determinantal varieties in the symmetric case. This can be constructed by means of classical invariant theory, so we also review the construction. Finally, we connect the splitting rings to the Kempf collapsing construction discussed above. In \S\ref{sec:pegrass}, we discuss the periplectic Grassmannian and apply all of the results thus far to computing the cohomology of its structure sheaf.

Finally, we comment on possible future directions. The super Grassmannian considered in \cite{superres} and the periplectic Grassmannian considered in this article are just two examples of homogeneous supervarieties. Our articles suggest that the problem of computing the cohomology of the structure sheaf of a general homogeneous supervariety will have a rich structure and be tractable. Beyond that, based on preliminary calculations, it seems that a certain class of homogeneous bundles beyond the structure sheaf should be approachable via the ideas in our papers. In a different direction, we wish to highlight the first author's article \cite{sam-osp} which shows that the Tor groups for certain determinantal varieties carry a representation for the orthosymplectic Lie superalgebra. The main difference there is that the ambient ring is no longer a polynomial ring, but rather a certain complete intersection ring, so that these representations are infinite-dimensional. It would be great if these results can be connected to the methods discussed above.

Furthermore, \cite{bwfact} shows that, in a certain stable range, the cohomology of Schur functors applied to the tautological sub and quotient bundles of the super Grassmannian can be expressed as a free module over the cohomology of the structure sheaf. Work in progress \cite{bwfact2} shows that the same is true for orthosymplectic Grassmannians. This suggests that many other cohomology calculations over homogeneous supervarieties could be tractable as well. 

\section{Signed splitting rings} \label{sec:typeB}

\subsection{Signed splitting rings} \label{ss:B-split}

Let $A$ be a ring and let $f=\sum_{i=0}^n a_{2n-2i} u^{2i}$ be a monic polynomial in $u^2$ with coefficients in $A$ (so $a_0=1$). We define the \defn{signed splitting ring} of $f$, denoted $\SSplit_A(f)$, to be the quotient $A[\eta_1, \ldots, \eta_n] / I$, where $I$ is the ideal generated by equating the coefficients of $f(u)=\prod_{i=1}^n (u^2-\eta_i^2)$. Explicitly, $I$ is generated by
\[
  a_{2i} - (-1)^ie_i(\eta_1^2, \ldots, \eta_n^2),
\]
where $e_i$ is the $i$th elementary symmetric polynomial. If $A$ is graded and $a_{2i}$ has degree $2i$ then $\SSplit_A(f)$ is graded with $\eta_i$ of degree~1. The hyperoctahedral group $\fW_n=\fS_n \ltimes (\bZ/2)^n$ acts $A$-linearly on $\SSplit_A(f)$, with $\fS_n$ permuting the $\eta_i$'s and the $i$th copy of $\bZ/2$ acting by $\pm 1$ on $\eta_i$.

Formation of the signed splitting ring is compatible with base change: if $A \to A'$ is a homomorphism, and $f'$ is the image of $f$ under $A[u] \to A'[u]$, then we have a natural isomorphism
\[
  \SSplit_{A'}(f')=A' \otimes_A \SSplit_A(f).
\]
In what follows, we let $A$ be a noetherian ring and put $B=\SSplit_A(f)$. We let $\Delta \in A$ be the discriminant of $f(u)$, which is $4^n \eta_1^2 \cdots \eta_n^2 \prod_{i < j} (\eta_i^2 - \eta_j^2)^2$. Since it comes up in the next statement, note that this implies that 2 is invertible if $\Delta$ is a unit.

Let $\tilde{f} = \sum_{i=0}^n a_{2n-2i} v^i$ be the monic polynomial in $v$ such that $\tilde{f}(u^2) = f(u)$.

We recall from \cite{superres} that the \defn{splitting ring} of $\tilde{f}$, denoted $\tilde{B}=\Split_A(\tilde{f})$, is the quotient $A[\xi_1,\dots,\xi_n]/I$, where $I$ is the ideal generated by equating the coefficients of $\tilde{f}(u) = \prod_{i=1}^n (u-\xi_i)$. Explicitly, $I$ is generated
\[
  a_{2i} - (-1)^i e_i(\xi_1,\dots,\xi_n).
\]
Note that we also have
\[
  B = \tilde{B}[\eta_1,\dots,\eta_n] / (\eta_1^2 - \xi_1,\dots, \eta_n^2- \xi_n).
\]

\subsection{Basic results}
Below, we recall that a ring homomorphism is \defn{syntomic} if it is flat, of finite presentation, and all of its fibers are locally complete intersection rings.

\begin{proposition} \label{prop:sign-split}
We have the following:
\begin{enumerate}
\item As an $A$-module, $B$ is free of rank $2^n n!$.
\item The map $A \to B$ is syntomic.
\item If $A$ satisfies Serre's condition $(S_k)$, then so does $B$. In particular, if $A$ is Cohen--Macaulay, then so is $B$.
\item If $\Delta$ is a unit of $A$ then $A \to B$ is \'etale.
\item If $A$ is reduced and $\Delta$ is a non-zerodivisor then $B$ is reduced.
\end{enumerate}
\end{proposition}

\begin{proof}
From the presentation $B = \tilde{B}[\eta_1,\dots,\eta_n] / (\eta_1^2 - \xi_1,\dots, \eta_n^2- \xi_n)$, we can interpret $A \to B$ as an iterated splitting ring: one for $\tilde{f}$, then $n$ more extensions for the degree $2$ polynomials $u^2 - \xi_i$. Hence all of these properties follow from \cite[Proposition 3.1]{superres} noting that both the syntomic and \'etale properties are preserved under taking composition.
\end{proof}

\begin{proposition} \label{prop:split-rep}
Suppose that $2$ and $n!$ are invertible in $A$. Then $B$ is free of rank one as an $A[\fW_n]$-module.
\end{proposition}

\begin{proof}
  The proof is similar to \cite[Proposition 3.2]{superres}.
\end{proof}

We pause to give a geometric source of splitting rings (see \cite[Theorem 8.2]{GSS}). Let $X$ be a smooth variety over an algebraically closed field and let $\cE$ be a rank $2n$ vector bundle on $X$ equipped with a symplectic form $\bigwedge^2 \cE \to \cO_X$, i.e., an alternating $2$-form which is non-degenerate at each fiber. Let $A$ be the Chow ring of $X$ and let $a_{2i} = c_{2i}(\cE)$, where $c_{2i}(\cE)$ is the $2i$th Chern class of $\cE$ (the existence of the symplectic form forces the odd Chern classes to vanish). Then the signed splitting ring $B$ of the polynomial
\[
  f(u) = \sum_{i=0}^n c_{2n-2i}(\cE) u^{2i}
\]
is the Chow ring of the relative isotropic flag variety $\bI\Fl(\cE)$.

Informally, $\bI\Fl(\cE)$ is a variety with a map to $X$ such that the fiber over $x \in X$ is the set of increasing sequences of subspaces
\[
  R_1 \subset R_2 \subset \cdots \subset R_n \subset \cE|_x
\]
such that the restriction of the symplectic form to $R_n$ is identically $0$ (i.e., $R_n$ is an isotropic subspace). More formally, consider the projective bundle $\pi \colon \bP(\cE) \to X$ with its tautological sequence
\[
  0 \to \cR \to \pi^* \cE \to \cO_{\bP(\cE)}(1) \to 0.
\]
Then the symplectic form pulls back to one on $\pi^* \cE$ and $\cR$ contains its orthogonal complement $\cR^\perp$, which is a rank 1 subbundle. If $\rank \cE = 2$, then define $\bI\Fl(\cE) = \bP(\cE)$. Otherwise, we inductively define $\bI\Fl(\cE) = \bI\Fl(\cR/\cR^\perp)$.

From the above definition of $\bI\Fl(\cE)$, it follows that the pullback of $\cE$ has a maximal flag of isotropic subbundles (i.e., whose successive quotients are line bundles and the rank of the biggest one is $n$), and the Chern classes of these line bundles are identified with the $\eta_i$. 

An important case for us is when $X = \Spec(\bC)$ and $\cE=\bC^{2n}$, so that $f=u^{2n}$. In that case, this discussion gives the following result (we note that the Chow ring and singular cohomology ring of $\bI\Fl(\bC^{2n})$ are isomorphic since it has a cellular decomposition).

\begin{proposition} \label{prop:split-flag}
Suppose that $f=u^{2n}$. Regard $A$ as graded and concentrated in degree~$0$, and $B$ as graded with each $\eta_i$ of degree~$2$. Then we have a natural isomorphism of graded rings
\begin{displaymath}
B = A \otimes \rH^*_{\sing}(\bI\Fl(\bC^{2n}), \bZ).
\end{displaymath}
\end{proposition}

\begin{remark}
  The signed splitting ring, as we have defined it, does not behave well when $2$ is not invertible. This is an artifact of using the quadratic extensions $\eta_i^2 - \xi_i$ to obtain $B$ from $\tilde{B}$. Instead, one should allow more general quadratic polynomials which are not forced to have repeated roots over a field of characteristic 2. However, since we are mostly interested in the characteristic 0 situation, we will not pursue this level of generality.
\end{remark}

\subsection{Normality criterion} \label{ss:b-normal}

Consider the following $n+\binom{n}{2}$ equations on $\Spec(B)$:
\begin{itemize}
\item $\eta_i=0$ for some $i=1,\dots,n$.
\item $\eta_i^2 = \eta_j^2$ for some $i \ne j$.
\end{itemize}
Let $\ol{E} \subset \Spec(B)$ be the locus where at least two of these equations (not necessarily from different bullet points) vanish, and let $E \subset \Spec(A)$ be the image of $\ol{E}$. Note that $E$ is closed since $\ol{E}$ is closed and $A \to B$ is finite. We say that an element $f$ of a normal ring $R$ is {\bf squarefree} if $v_{\fp}(f) \in \{0,1\}$ for all height one primes $\fp$ of $R$, where $v_{\fp}$ denotes the valuation associated to $\fp$. We say that a subset of $\Spec(A)$ has \defn{codimension $\ge c$} if all primes it contains have height $\ge c$.

We will assume that $2$ is invertible in $A$ in this section (this is forced by the assumptions in the following results).

\begin{proposition} \label{prop:sign-norm-crit}
Suppose the following conditions hold:
\begin{enumerate}
\item $A$ is normal,
\item $\Delta$ is squarefree and a non-zerodivisor,
\item $E$ has codimension $\ge 2$ in $\Spec(A)$.
\end{enumerate}
Then $B$ is normal.
\end{proposition}

\begin{proof}
  Let $\tilde{E} \subset \Spec(\tilde{B})$ be the locus where at least two of the conditions $\xi_i = \xi_j$ for $i \ne j$ are satisfied. This is contained in the image of $\ol{E}$ under $\Spec(B) \to \Spec(\tilde{B})$, and so by (c), the image of $\tilde{E}$ in $\Spec(A)$ has codimension $\ge 2$. Let $\tilde{\Delta}$ be the discriminant of $\tilde{f}$. Then $\Delta = 4^n a_{2n} \tilde{\Delta}$, and hence $\tilde{B}$ is normal by \cite[Proposition 3.5]{superres}.

  We claim that $a_{2n}$ has valuation $0$ or $1$ for any prime of $\tilde{B}$. Let $\fq$ be a prime of $\tilde{B}$ lying over a prime $\fp$ in $A$. First, since $\Delta$ is squarefree, either $a_{2n}$ is a unit in $A_\fp$, or $\tilde{\Delta}$ is a unit in $A_\fp$. The claim is immediate in the first case. In the second case, $A_\fp \to \tilde{B}_\fq$ is \'etale by \cite[Proposition 3.1(d)]{superres}, and so the valuation $v_\fq$ on $\tilde{B}_\fq$ restricts to the valuation $v_\fp$ on $A_\fp$, and hence the claim follows.

  Finally, we deduce that $B$ is normal by applying \cite[Proposition 3.5]{superres} again (using that $B$ is an iterated splitting ring starting from $\tilde{B}$).
\end{proof}

We now give a variant of Proposition~\ref{prop:sign-norm-crit}. For $\Delta \in A$, define $V(\Delta, \partial \Delta) \subset \Spec(A)$ to be the set of points $x \in \Spec(A)$ at which $\Delta$ vanishes to order two, in the sense that its image under $A \to A_x$ belongs to $\fm_x^2$, where $\fm_x$ is the maximal ideal of $A_x$. If $A$ is finitely generated over a field $k$ and $x$ is a smooth point of $\Spec(A)$, then $x$ belongs to $V(\Delta, \partial \Delta)$ if and only if $\Delta=0$ in the residue field $\kappa(x)$ and $d\Delta=0$ in $\Omega^1_{A/k} \otimes_A \kappa(x)$; since $\Omega^1_{A/k}$ is locally free on the smooth locus, this shows that $V(\Delta, \partial \Delta)$ is closed in the smooth locus.

\begin{lemma} \label{lem:E-partialDelta}
We have $E \subset V(\Delta, \partial \Delta)$.
\end{lemma}

\begin{proof}
  Let $x \in E$, and let $n-c = \#\{\eta_1^2,\dots,\eta_n^2\}$; let $c'=0$ if all $\eta_1,\dots,\eta_n$ are non-zero in the residue field, and 1 otherwise. By definition of $E$, we have $c+c' \ge 2$. By \cite[Lemma 3.8]{superres}, we have $\tilde{\Delta} \in \fm_x^c$, and by definition, $a_{2n} \in \fm_x^{c'}$. Since $\Delta = 4^n a_{2n} \tilde{\Delta}$, we see that $\Delta \in \fm_x^2$, so $x \in V(\Delta, \partial \Delta)$.
\end{proof}

\begin{proposition} \label{prop:sign-norm-crit2}
Suppose the following conditions hold:
\begin{enumerate}
\item $A$ is normal,
\item $\Delta$ is a non-zerodivisor,
\item $V(\Delta, \partial \Delta)$ has codimension $\ge 2$.
\end{enumerate}
Then $\Delta$ is squarefree and $B$ is normal.
\end{proposition}

\begin{proof}
  We apply Proposition~\ref{prop:sign-norm-crit}. The set $E$ there has codimension $\ge 2$ by the present assumption (c) and Lemma~\ref{lem:E-partialDelta}. It thus suffices to prove that $\Delta$ is squarefree.  Let $\fp$ be a height one prime of $A$ so $A_\fp$ is a DVR by (a). By (c), $\fp \notin V(\Delta, \partial \Delta)$, so that $\Delta \notin \fp^2 A_\fp$. In particular, this means $v_\fp(\Delta) \le 1$.
\end{proof}

\subsection{Signed factorization rings}

let $f=\sum_{i=0}^n a_{2n-2i} u^{2i}$ be a monic polynomial over a ring $A$. Let $p$ and $q$ be non-negative integers such that $p+q=n$, and put $g=\sum_{i=0}^p b_{p-i} u^i$ and $h=\sum_{i=0}^q c_{2q-2i} u^{2i}$, where $b_0=c_0=1$ and the remaining $b_i$ and $c_i$ are formal symbols. We define the \defn{signed $(p,q)$-factorization ring} of $f$, denoted $\SFact^{p,q}_A(f)$ to be $A[b_1, \ldots, b_p, c_2, \ldots, c_{2q}]/I$, where $I$ is the ideal generated by equating the coefficients of
\[
  f(u)=g(u) g(-u) h(u).
\]
If $A$ is graded and $a_i$ is homogeneous of degree $2i$ then $\SFact^{p,q}_A(f)$ is graded and $\deg(b_i)=i$ and $\deg(c_{2i})=2i$. Formation of the signed factorization ring is compatible with base change, as with the signed splitting ring.

In what follows, we let $B=\SSplit_A(f)$ and $C=\SFact^{p,q}_A(f)$.

\begin{proposition} \label{prop:sign-fact}
We have the following:
\begin{enumerate}
\item   We have a natural $A$-algebra isomorphism $B=\Split_C(g) \otimes_C \SSplit_C(h)$.  
\item As an $A$-module, $C$ is free of rank $2^p \binom{n}{p}$.
\item The map $A \to C$ is syntomic.
\item If $A$ satisfies Serre's condition $(S_k)$, then so does $C$. In particular, if $A$ is Cohen--Macaulay, then so is $C$.
\item If $B$ is reduced (resp., integral, normal) then so is $C$.
\end{enumerate}
\end{proposition}

\begin{proof}
(a)  Let $\eta'_1, \ldots, \eta'_p$ be the generators of $\Split_C(g)$ and $\eta'_{p+1}, \ldots, \eta'_{p+q}$ those for $\SSplit_C(h)$. Put $B'=\Split_C(g) \otimes_C \SSplit_C(h)$. Since $f(u) = \prod_{i=1}^n (u^2-{\eta'_i}^2)$ holds over $B'$, we have an $A$-algebra homomorphism $\phi \colon B \to B'$ given by $\phi(\eta_i)=\eta'_i$. Let $g^*(u)=\prod_{i=1}^p(u-\eta_i)$ and $h^*(u)=\prod_{i=p+1}^n(u^2-\eta^2_i)$ be polynomials in $B[u]$. The factorization
  \[
    f(u)=g^*(u) g^*(-u) h^*(u)
  \]
  gives an $A$-algebra homomorphism $C \to B$ mapping $g(u)$ to $g^*(u)$ and $h(u)$ to $h^*(u)$. The tautological splittings of $g^*(u)$ and $h^*(u)$ over $B$ yield an $A$-algebra homomorphism $\psi \colon B' \to D$ given by $\psi(\eta'_i)=\eta_i$. Since $\phi$ and $\psi$ are clearly inverses, the result follows.

  (b) By Proposition~\ref{prop:sign-split}(a), we have an $A$-module isomorphism $B \cong A^{\oplus 2^n n!}$ and $C$-module isomorphisms $\Split_C(g) \cong C^{\oplus p!}$ and $\SSplit_C(h) \cong C^{\oplus 2^q q!}$. Comparing with (a), we obtain an $A$-module isomorphism $C^{\oplus 2^q p! q!} \cong A^{\oplus 2^n n!}$. It follows that $C$ is projective as an $A$-module of constant rank $2^p \binom{n}{p}$. To finish, we consider the universal case: let
  \[
    A^\univ = \bZ[a_2,\dots,a_{2n}], \qquad f^\univ(u) = u^{2n} + \sum_{i=0}^{n-1} a_{2(n-i)} u^{2i}, \qquad C^\univ = \SFact^{p,q}_{A^\univ}(f^\univ).
  \]
  Applying the previous discussion, $C^\univ$ is a projective graded $A^\univ$-module, and hence must be free. By base change, we see that $C$ must be free over $A$.

(c) Let $\fp$ be a prime of $A$. Then $C \otimes_A \kappa(\fp)$ is finite over $\kappa(\fp)$ by (b), and therefore of Krull dimension~0. This ring is a quotient of $\kappa(\fp)[b_1,\dots,b_p,c_2,\dots,c_{2q}]$ by $p+q$ relations, and is therefore a complete intersection. Thus $A \to C$ is syntomic.

(d) This follows since the property is syntomic local.

(e) From (a), $C$ is isomorphic to a subring of $B$, which handles the reduced and integral conditions. For the normality condition, we use that $B$ is a splitting ring over a signed splitting ring over $C$ by (a), and hence the inclusion $C \to B$ admits a $C$-linear splitting by Proposition~\ref{prop:sign-split}(f) and \cite[Proposition 3.1(f)]{superres}.
\end{proof}

Just like the signed splitting rings, we have a geometric source of signed factorization rings (see \cite[Theorem 8.2]{GSS}). Let $X$ be a smooth variety over an algebraically closed field and let $\cE$ be a rank $2n$ vector bundle on $X$ equipped with a symplectic form $\bigwedge^2 \cE \to \cO_X$. Let $A$ be the Chow ring of $X$ and let $a_{2i} = c_{2i}(\cE)$, where $c_{2i}(\cE)$ is the $2i$th Chern class of $\cE$. Then the signed $(p,q)$-factorization ring $C$ is the Chow ring of the relative Grassmannian $\bI\Gr_p(\cE)$ of rank $p$ isotropic subbundles of $\cE$. Furthermore, on $\bI\Gr_p(\cE)$, the pullback of $\cE$ has a rank $p$ isotropic subbundle $\cR$; we have a filtration $\cR \subset \cR^\perp \subset \cE$ where $\cR^\perp$ is the orthogonal complement of $\cR$ in $\cE$ with respect to the symplectic form. Then $g(u)$ is the Chern polynomial of $\cR$ and $h(u)$ is the Chern polynomial of $\cR^\perp / \cR$.

An important case for us is when $X = \Spec(\bC)$ and $\cE=\bC^{2n}$, so that $f=u^{2n}$. In that case, this discussion gives the following result (we note that the Chow ring and singular cohomology ring of $\bI\Gr_p(\bC^{2n})$ are isomorphic since it has a cellular decomposition).

\begin{proposition} \label{prop:sign-fact-grass}
Suppose that $f=u^{2n}$. Regard $A$ as graded and concentrated in degree~$0$, and $B$ as graded with each $\eta_i$ of degree~$2$. Then we have a natural isomorphism of graded rings
\begin{displaymath}
\SFact^{p,q}_A(f) = A \otimes \rH^*_{\sing}(\bI\Gr_p(\bC^{2n}), \bZ).
\end{displaymath}
\end{proposition}

\begin{remark}
  One can also form partial signed splitting rings which are intermediate between $B$ and $C$, and all of the above properties generalize. In that case, we get an isomorphism with $A$ tensored with the cohomology ring of the corresponding isotropic partial flag variety.
\end{remark}

\section{Type D splitting rings} \label{sec:typeD}

We will also need a slight variation of the signed splitting ring which takes into account cases in which the constant term of $f$ is already a square in the base ring $A$. Roughly speaking, our main example will come from the fact that the determinant of a skew-symmetric matrix is the square of its Pfaffian, but that will be discussed later. While many things will be similar, it is not clear to us how to deduce the main results from what we have already shown, so we will have to redo some proofs from \cite{superres} with the appropriate modifications. We call them ``type D'' to be consistent with the Coxeter groups of type D (in the sense of Dynkin diagrams).

\subsection{Type D splitting rings}

The setup is as before: $A$ is a ring and $f=\sum_{i=0}^n a_{2n-2i} u^{2i}$ is a monic polynomial in $u^2$ with coefficients in $A$ (so $a_0=1$). Additionally, suppose we have an element $\alpha \in A$ such that $\alpha^2 = a_{2n}$.

We define the \defn{type D splitting ring} of $f$, denoted $\DSplit_A(f, \alpha)$, to be the quotient $A[\eta_1, \ldots, \eta_n] / I$, where $I$ is the ideal generated by equating the coefficients of $f(u)=\prod_{i=1}^n (u^2-\eta_i^2)$ {\it and} the equation $\eta_1\cdots \eta_n = \alpha$. Explicitly, $I$ is generated by
\[
  a_{2i} - (-1)^ie_i(\eta_1^2, \ldots, \eta_n^2), \quad (\text{for } i=1,\dots,n-1), \qquad \eta_1\cdots \eta_n = \alpha.
\]
where $e_i$ is the $i$th elementary symmetric polynomial.

If $A$ is graded and $a_{2i}$ has degree $2i$, and $\deg(\alpha)=n$, then $\DSplit_A(f, \alpha)$ is graded with $\eta_i$ of degree~2. Consider the homomorphism $\fS_n \ltimes (\bZ/2)^n \to \bZ/2$ given by $(\sigma, z_1,\dots,z_n) \mapsto z_1+\cdots+z_n$ and let $\fW_n$, the demihyperoctahedral group, be its kernel. Then $\fW_n$ acts $A$-linearly on $\DSplit_A(f, \alpha)$, with $\fS_n$ permuting the $\eta_i$'s and the $i$th copy of $\bZ/2$ acting by $\pm 1$ on $\eta_i$.

Formation of the signed splitting ring is compatible with base change: if $A \to A'$ is a homomorphism,  $\alpha'$ the image of $\alpha$, and $f'$ is the image of $f$ under $A[u] \to A'[u]$, then we have a natural isomorphism
\[
  \DSplit_{A'}(f', \alpha')=A' \otimes_A \DSplit_A(f, \alpha).
\]

\begin{remark}
  The symmetric and (demi)hyperoctahedral groups are special cases of the complex reflection groups $G(m,p,n)$; in this context, there is a general construction that encompasses both the signed splitting ring and the type D splitting ring. Namely, pick positive integers $n,p,m$ such that $p$ divides $m$. Given a ring $A$ and a monic polynomial $f(u) = u^{mn} + \sum_{i=0}^{n-1} a_{m(n-i)} u^{mi}$ and an element $\alpha \in A$ such that $\alpha^p = a_{mn}$, we define the generalized splitting ring to be the quotient of $A[\eta_1,\dots,\eta_n]$ by the ideal $I$ generated by
  \[
    a_{mi} - (-1)^i e_i(\eta_1^m, \dots, \eta_n^m), \qquad (\text{for $i=1,\dots,n-1$}), \qquad (\eta_1\cdots \eta_n)^{m/p} = \alpha.
  \]
  Then the splitting ring of \cite{superres} corresponds to the case $m=p=1$, the signed splitting ring corresponds to $m=2$ and $p=1$, and the type D splitting ring corresponds to $m=p=2$. Analogues of the results that we have proven should generalize in a straightforward way, but we leave the details to the interested reader.

  However, it would be of great interest if the geometric applications of signed and type D splitting rings that we establish later in this paper have analogues for this more general class of rings.
\end{remark}

\subsection{The universal case}

Let $A^{\univ}=\bZ[a_2, \ldots, a_{2n-2}, \tilde{\alpha}]$, and define
\[
  f^{\univ}(u) = u^{2n}+\sum_{i=1}^{n-1} a_{2n-2i} u^{2i} + \tilde{\alpha}^2 \in A^\univ[u],\qquad  B^{\univ} = \DSplit_{A^{\univ}}(f^{\univ}, \tilde{\alpha}).
\]
The map $\bZ[\eta_1, \ldots, \eta_n] \to B^{\univ}$ is surjective, and has no kernel since $B^{\univ} \otimes \bC$ clearly has Krull dimension $n$. Thus we have
\[
  B^{\univ}=\bZ[\eta_1, \ldots, \eta_n].
\]

\begin{proposition} \label{prop:univ-free}
  As an $A^\univ$-module, $B^\univ$ is free of rank $2^{n-1} n!$.
\end{proposition}

\begin{proof}
  The map $A^\univ \to B^\univ$ is finite, as each $\eta_i$ is a root of $f^{\univ}$. Since $A^\univ \to B^\univ$ is a finite map of polynomial rings of the same dimension, it is flat \stacks{00R4}. Therefore $B^\univ$ is projective as an $A^\univ$-module, and thus free, as any projective graded $A^\univ$-module is free. The rank can be computed over the fiber of the ideal $(a_2,\dots,a_{2n-2},\alpha)$; in this case we have the ring $\bQ[\eta_1,\dots,\eta_n] / I$ where $I$ is the ideal generated by $e_i(\eta_1^2,\dots,\eta_n^2)$ for $i=1,\dots,n-1$ and $\eta_1\cdots\eta_n$. This ideal is a graded complete intersection, and hence the rank follows by taking the product of the degrees of its minimal generators.
\end{proof}

Given our original setup, there is unique ring homomorphism $A^{\univ} \to A$ such that $f$ and $\alpha$ are the images of $f^\univ$ and $\tilde{\alpha}$, respectively. Since formation of type D splitting rings is compatible with base change, we have $B=A \otimes_{A^{\univ}} B^{\univ}$.

\subsection{Basic results}

In what follows, we let $A$ be a noetherian ring and put $B=\DSplit_A(f, \alpha)$. Let $\tilde{\Delta}$ be the discriminant of $\sum_{i=0}^n a_{2n-2i} v^i$. We define the \defn{reduced discriminant} of the pair $(f,\alpha)$ by
\[
  \ol{\Delta} = \alpha \tilde{\Delta},
\]
which is an element of $A$. In $B$, we have the formula
\[
  \ol{\Delta} = 4^n \eta_1 \cdots \eta_n \prod_{i < j} (\eta_i^2 - \eta_j^2)^2.
\]
Note that $\alpha \ol{\Delta}$ is the usual discriminant of $f$.

\begin{proposition} \label{prop:d-sign-split}
We have the following:
\begin{enumerate}
\item As an $A$-module, $B$ is free of rank $2^{n-1} n!$.
\item The map $A \to B$ is syntomic.
\item If $A$ satisfies Serre's condition $(S_k)$, then so does $B$. In particular, if $A$ is Cohen--Macaulay, then so is $B$.
\item If $\ol{\Delta}$ is a unit of $A$ then $A \to B$ is \'etale.
\item If $A$ is reduced and $\ol{\Delta}$ is a non-zerodivisor then $B$ is reduced.
\end{enumerate}
\end{proposition}

\begin{proof}
  (a) This follows from Proposition~\ref{prop:univ-free} and base change.
  
(b) Suppose that $\fp$ is a prime of $A$. Then $B \otimes_A \kappa(\fp)$ is finite over $\kappa(\fp)$ by~(a), and therefore of Krull dimension~0. This ring is a quotient of $\kappa(\fp)[\eta_1,\dots,\eta_n]$ by $n$ relations, and is therefore a complete intersection. It follows that $A \to B$ is syntomic.

(c) Since $(S_k)$ is syntomic local \stacks{036A}, the result follows from (b).

(d) We have $0=f(\eta_i)$ and so $0=f'(\eta_i) d\eta_i$. However, $f'(\eta_i)=2 \eta_i \prod_{j \ne i} (\eta_i-\eta_j)$ divides $\ol{\Delta}$ and is therefore a unit. Thus $d\eta_i=0$. We conclude that $\Omega_{B/A}=0$. Since $B$ is finite flat over $A$ by (a), it is therefore \'etale.

(e) Since $A$ is reduced, it satisfies $(R_0)$ and $(S_1)$ \stacks{031R}. Thus $B$ satisfies $(S_1)$ by part~(c). Since $V(\ol{\Delta}) \subset \Spec(A)$ has codimension~1 and $A[1/\ol{\Delta}] \to B[1/\ol{\Delta}]$ is \'etale, it follows that $B$ satisfies $(R_0)$. Thus $B$ is reduced.
\end{proof}

\begin{proposition} \label{prop:d-split-rep}
Suppose that $2$ and $n!$ are invertible in $A$. Then $B$ is free of rank one as an $A[\fW_n]$-module.
\end{proposition}

\begin{proof}
  The proof is similar to \cite[Proposition 3.2]{superres}.
\end{proof}

We pause to give a geometric source of splitting rings (see \cite[Theorem 8.2]{GSS}). Let $X$ be a smooth variety over an algebraically closed field and let $\cE$ be a rank $2n$ vector bundle on $X$ equipped with a orthogonal form $\Sym^2 \cE \to \cO_X$, i.e., a symmetric $2$-form which is non-degenerate at each fiber such that the pullback of $\cE$ to its flag bundle has an isotropic subbundle of rank $n$.

Let $A$ be the Chow ring of $X$ with 1/2 adjoined, and let $a_{2i} = c_{2i}(\cE)$, where $c_{2i}(\cE)$ is the $2i$th Chern class of $\cE$ (the existence of the orthogonal form forces the odd Chern classes to vanish). We can construct the type D splitting ring $B$ of the polynomial $f = \sum_{i=0}^n c_{2n-2i}(\cE) u^{2i}$ by taking $\alpha$ to be the Euler class of $\cE$ in the sense of \cite[\S 4]{edidin-graham}.
We let $\bO\Fl(\cE)$ denote the relative orthogonal flag variety, which we will define fiberwise; the formal definition can be given as in \S\ref{ss:B-split}. Over $x \in X$, points of $\bO\Fl(\cE)$ are tuples $(F_1 \subset \cdots \subset F_{n-1}, F_n, F'_n)$ where $F_1 \subset \cdots \subset F_{n-1}$ is a flag of $n-1$  subspaces of $\cE|_x$ as usual, and $F_n,F'_n$ are rank $n$ isotropic subspaces of $\cE|_x$ such that $F_n \cap F'_n = F_{n-1}$. Then $\DSplit_A(f,\alpha)$ is the Chow ring of $\bO\Fl^+(\cE)$ with 1/2 adjoined \cite[\S 6, Theorem 6]{edidin-graham}. 

An important case for us is when $X = \Spec(\bC)$ and $\cE=\bC^{2n}$, so that $f=u^{2n}$ and $\alpha=0$. In that case, this discussion gives the following result (we note that the Chow ring and singular cohomology ring of $\bO\Fl(\bC^{2n})$ are isomorphic since it has a cellular decomposition).

\begin{proposition} \label{prop:d-split-flag}
Suppose that $f=u^{2n}$ and $\alpha=0$ and $1/2 \in A$. Regard $A$ as graded and concentrated in degree~$0$, and $B$ as graded with each $\eta_i$ of degree~$2$. Then we have a natural isomorphism of graded rings
\begin{displaymath}
B = A \otimes_{\bZ[1/2]} \rH^*_{\sing}(\bO\Fl(\bC^{2n}), \bZ[1/2]).
\end{displaymath}
\end{proposition}

\subsection{Normality criterion}

We now turn our attention to the question of when $B$ is normal.

Consider the following $n+\binom{n}{2}$ equations on $\Spec(B)$:
\begin{itemize}
\item $\eta_i=0$ for some $i=1,\dots,n$.
\item $\eta_i^2 = \eta_j^2$ for some $i \ne j$.
\end{itemize}
Let $\tilde{E} \subset \Spec(B)$ be the locus where at least two of these equations (not necessarily from different bullet points) vanish, and let $E \subset \Spec(A)$ be the image of $\tilde{E}$. Note that $E$ is closed since $\tilde{E}$ is closed and $A \to B$ is finite.

We will assume that $2$ is invertible in $A$ in this section (this is forced by the assumptions in the following results).

\begin{proposition} \label{prop:d-norm-crit}
Suppose the following conditions hold:
\begin{enumerate}
\item $A$ is normal,
\item $\ol{\Delta}$ is squarefree and a non-zerodivisor,
\item $E$ has codimension $\ge 2$ in $\Spec(A)$.
\end{enumerate}
Then $B$ is normal.
\end{proposition}

\begin{proof}
  First suppose that $A$ is a strictly henselian discrete valuation ring with maximal ideal $\fm$ and residue field $A/\fm = \kappa$. We show that $B$ is regular. If $\ol{\Delta}$ is a unit of $A$ then $B$ is \'etale over $A$ and thus regular. Assume then that $\ol{\Delta}$ is not a unit; by hypothesis (b), it is a uniformizer of $A$. Let $\ol{f}$ be the image of $f$ in $\kappa[u]$. Since $\ol{\Delta}$ maps to~0 in $\kappa$ it follows that $\ol{f}$ has a repeated root; by hypothesis (c), there are two possibilities:
  \begin{itemize}
  \item 0 is a root with multiplicity exactly 2, and the rest of the $2n-2$ roots are distinct, or
  \item there are two repeated roots of the form $x,x,-x,-x$ with $x \ne 0$, and the rest of the roots are nonzero and distinct.
  \end{itemize}

We consider the cases separately.

{\bf Case 1: }
  We have a factorization $\ol{f}(u) = u^2 \cdot \ol{g}(u^2)$ over $\kappa$, where $\ol{g}(u)$ has $n-1$ distinct non-zero roots over the algebraic closure $\ol{\kappa}$.

  Since $\kappa$ is separably closed, it follows that $\ol{g}(u^2) = (u^2-\ol{x}^2_2) \cdots (u^2-\ol{x}_n^2)$ splits completely. By the henselian property, we thus have a factorization
  \[
    f(u)= (u^2 - c) (u^2-x_2^2) \cdots (u^2-x_n^2),
  \]
  where $x_i \in A$ lifts $\ol{x}_i$ for $i \ge 2$ and $c \in \fm$.

Now let $\fp$ be a prime of $B$ above the maximal ideal of $A$, and work in $B_{\fp}$ in what follows. Applying a permutation if necessary, we can assume that $\ol{\eta}_i=\ol{x}_i$ for $i \ge 2$, where $\ol{\eta}_i$ is the image of $\eta_i$ in $B_\fp/\fp \cong \kappa$. For $i \ge 2$, it follows that $\eta_i^2 - c$ and $\eta_i-x_j$, for $j \ne i$, are non-zero in $\kappa$, and thus units of $B_\fp$; since $f(\eta_i)=0$, we conclude that $\eta_i=x_i$. This shows that $B_{\fp}$ is generated as an $A$-algebra by $\eta_1$. We have
\[
  (u^2-c)\prod_{i \ge 2} (u^2-\eta_i^2) = \prod_{i \ge 1} (u^2-\eta_i^2).
\]
Since monic polynomials are non-zerodivisors, it follows that $c = \eta^2_1$ in $B_\fp$, i.e., $B_\fp \cong A[y]/(y^2-c)$. Finally, $c$ divides $\ol{\Delta}$, so is a uniformizer of $A$, and hence we see that $y$ generates the maximal ideal of $B_\fp$, which shows that $B_\fp$ is also a DVR.

\medskip

  {\bf Case 2: }
In this case, we have a factorization $\ol{f}(u) = \ol{q}(u^2) \cdot \ol{g}(u^2)$ over $\kappa$, where $\ol{q}(u)$ is a quadratic polynomial with a non-zero repeated root, and $\ol{g}(u)$ has distinct non-zero roots over the algebraic closure $\ol{\kappa}$ (which are also distinct from the roots of $\ol{q}(u)$).

Since $\kappa$ is separably closed, it follows that $\ol{g}(u^2) = (u^2-\ol{x}_3^2) \cdots (u^2-\ol{x}^2_n)$ splits completely. By the henselian property, we thus have a factorization $f(u)= q(u^2) (u^2-x^2_3) \cdots (u^2-x^2_n)$, where $x_i \in A$ lifts $\ol{x}_i$ for $i \ge 3$ and $q(u)$ is a quadratic polynomial with coefficients in $A$.

  Now let $\fp$ be a prime of $B$ above the maximal ideal of $A$, and work in $B_{\fp}$ in what follows. As in the previous case, and applying a permutation if necessary, we can conclude that $\eta_i = x_i$ for $i\ge 3$. This shows that $B_{\fp}$ is generated as an $A$-algebra by $\eta_1$ and $\eta_2$. We have
\[
  q(u^2) \prod_{i \ge 3} (u^2-\eta_i^2) = \prod_{i \ge 1} (u^2-\eta_i^2).
\]
Since monic polynomials are non-zerodivisors, it follows that $q(u^2) = (u^2-\eta^2_1)(u^2-\eta^2_2)$ in $B_\fp$. Furthermore, $x_i$ is a unit for $i \ge 3$ since its image in $\kappa$ is non-zero; set $\beta = \alpha  x_3^{-1}\cdots x_n^{-1}$. Write $q(u^2) = u^4 - au^2 + \beta^2$. Then
\[
  B_\fp \cong \DSplit_A(q(u^2), \beta) = A[y_1, y_2] / (y_1^2+y_2^2-a, y_1y_2 - \beta) \cong A[u] / (q(u^2)).
\]
To justify the last isomorphism, define a map $A[u] / q(u^2) \to \DSplit_A(q(u^2), \beta)$ by $u \mapsto y_1$. This map is surjective since $-u(u^2-a)/\beta \mapsto y_2$, and hence it is an isomorphism since both are free $A$-modules of rank 4. Next, $u$ is a unit in $B_\fp$ since $-u^2(u^2-a)/\beta^2 = 1$. Hence the maximal ideal of $B_\fp$ is $\fm B_\fp$, and so $B_\fp$ is a DVR.

We now treat the general situation. Let $\fp$ be a height one prime of $A$. We show that $B_{\fp}$ is regular. Let $A_{\fp}^{\rm sh}$ be the strict henselization of the DVR $A_{\fp}$. By the previous paragraphs, we see that $B_{\fp} \otimes_{A_{\fp}} A_{\fp}^{\rm sh}$ is regular. Now, $A_{\fp}^{\rm sh}$ is the direct limit of a family $\{A_i\}$ of rings, each of which is an \'etale cover of $A_{\fp}$. The above arguments apply with $A_i$ in place of $A_{\fp}^{\rm sh}$ for $i$ sufficiently large. We conclude that $B_{\fp} \otimes_{A_{\fp}} A_i$ is regular for some $i$. Since regularity is \'etale local, we conclude that $B_{\fp}$ is regular. Thus $B$ is regular in codimension~1. Finally, since $A$ is normal, it satisfies Serre's condition $(S_2)$, and hence the same is true for $B$ by Proposition~\ref{prop:d-sign-split}(c), so $B$ is therefore normal.
\end{proof}

We now give a variant of Proposition~\ref{prop:d-norm-crit}. For $\ol{\Delta} \in A$, define $V(\ol{\Delta}, \partial \ol{\Delta}) \subset \Spec(A)$ to be the set of points $x \in \Spec(A)$ at which $\ol{\Delta}$ vanishes to order two in the same sense as in \S\ref{ss:b-normal}.

\begin{lemma} \label{lem:d-E-partialDelta}
We have $E \subset V(\ol{\Delta}, \partial \ol{\Delta})$.
\end{lemma}

\begin{proof}
  The proof is similar to the proof of Lemma~\ref{lem:E-partialDelta}.
\end{proof}

\begin{proposition} \label{prop:d-sign-norm-crit2}
Suppose the following conditions hold:
\begin{enumerate}
\item $A$ is normal,
\item $\ol{\Delta}$ is a non-zerodivisor,
\item $V(\ol{\Delta}, \partial \ol{\Delta})$ has codimension $\ge 2$.
\end{enumerate}
Then $\ol{\Delta}$ is squarefree and $B$ is normal.
\end{proposition}

\begin{proof}
  We apply Proposition~\ref{prop:d-norm-crit}. The set $E$ there has codimension $\ge 2$ by the present assumption (c) and Lemma~\ref{lem:E-partialDelta}. It thus suffices to prove that $\ol{\Delta}$ is squarefree, and the proof of this is similar to the corresponding proof in Proposition~\ref{prop:sign-norm-crit2}.
\end{proof}

\subsection{Type D factorization rings}

As in the previous section, let $f(u)= u^{2n} + \sum_{i=1}^{n-1} a_{2n-2i} u^{2i} + \alpha^2$ be a monic polynomial in $u^2$ over a ring $A$ together with a choice of square root $\alpha$ for its constant term. Let $p$ and $q$ be non-negative integers such that $p+q=n$, and put $g=\sum_{i=0}^p b_{p-i} u^i$ and $h=\sum_{i=1}^{q} c_{2q-2i} u^{2i} + \beta^2$, where $b_0=c_0=1$ and the remaining $b_i$, $c_i$, and $\beta$ are formal symbols. We define the \defn{type D $(p,q)$-factorization ring} of the pair $(f, \alpha)$ by
\[
  \DFact^{p,q}_A(f, \alpha) = A[b_1, \ldots, b_p, c_2, \ldots, c_{2q-2}, \beta]/I,
\]
where $I$ is the ideal generated by equating the coefficients of
\[
  f(u)=g(u) g(-u) h(u), \qquad \alpha = b_p \beta.
\]
If $A$ is graded and $a_i$ is homogeneous of degree $2i$ and $\alpha$ homogeneous of degree $n$, then $\DFact^{p,q}_A(f)$ is graded with $\deg(b_i)=i$, $\deg(c_{2i})=2i$, and $\deg(\beta)=q$. Formation of the type D factorization ring is compatible with base change, as with the type D splitting ring.

In what follows, we let $B=\DSplit_A(f, \alpha)$ and $C=\DFact^{p,q}_A(f, \alpha)$.

\begin{proposition} \label{prop:d-sign-fact}
We have the following:
\begin{enumerate}
\item We have a natural $A$-algebra isomorphism $B=\Split_C(g) \otimes_C \DSplit_C(h, \beta)$.
\item As an $A$-module, $C$ is free of rank $2^p \binom{n}{p}$ if $p < n$, and is free of rank $2^{n-1}$ if $n=p$.
\item The map $A \to C$ is syntomic.
\item If $A$ satisfies Serre's condition $(S_k)$, then so does $C$. In particular, if $A$ is Cohen--Macaulay, then so is $C$.
\item If $B$ is reduced (resp., integral, normal) then so is $C$.
\end{enumerate}
\end{proposition}

\begin{proof}
This is similar to the proof of Proposition~\ref{prop:sign-fact}.
\end{proof}

Just like the type D splitting rings, we have a geometric source of type D factorization rings. Let $X$ be a smooth variety over an algebraically closed field and let $\cE$ be a rank $2n$ vector bundle on $X$ equipped with an orthogonal form $\Sym^2 \cE \to \cO_X$. Let $A$ be the Chow ring of $X$ with $1/2$ adjoined and let $a_{2i} = c_{2i}(\cE)$, where $c_{2i}(\cE)$ is the $2i$th Chern class of $\cE$. As before, let $\alpha$ be the Euler class of $\cE$, so that $\alpha^2 = c_{2n}(\cE)$.

If $p<n$, then the type D $(p,q)$-factorization ring $C$ is the Chow ring (with $1/2$ adjoined) of the relative Grassmannian $\bO\Gr_p(\cE)$; its fiber over $x \in X$ is the Grassmannian of rank $p$ isotropic subspaces of $\cE|_x$. For $p=n$, we instead have to consider the relative Grassmannian, whose fiber over $x \in X$ consists of (unordered) pairs of rank $n$ isotropic subspaces of $\cE|_x$ such that their intersection is a subbundle of rank $n-1$. We will call this $\bO\Gr_n(\cE)$.

\begin{remark}
  When $p=n-1$, $\bO\Gr_{n-1}$ is not the quotient of $\SO_{2n}(\bC)$ by a maximal parabolic subgroup. In terms of Bourbaki labeling of the type D Dynkin diagram, this corresponds to the parabolic subgroup that fixes both of the fundamental weights $\omega_{n-1}$ and $\omega_{n}$.
\end{remark}

An important case for us is when $X = \Spec(\bC)$ and $\cE=\bC^{2n}$, so that $f=u^{2n}$ (and $\alpha=0$). In that case, this discussion gives the following result (we note that the Chow ring and singular cohomology ring of $\bO\Gr_p(\bC^{2n})$ are isomorphic since it has a cellular decomposition).

\begin{proposition} \label{prop:d-sign-fact-grass}
Suppose that $1/2\in A$. Regard $A$ as graded and concentrated in degree~$0$, and $\DFact^{p,q}_A(u^{2n},0)$ as graded with each $\eta_i$ of degree~$2$. Then we have a natural isomorphism of graded rings
\begin{displaymath}
\DFact^{p,q}_A(u^{2n}, 0) = A \otimes_{\bZ[1/2]} \rH^*_{\sing}(\bO\Gr_p(\bC^{2n}), \bZ[1/2]).
\end{displaymath}
\end{proposition}

\section{Periplectic case: the determinantal variety} \label{s:pedetvar}

Throughout this section, $V$ will denote a complex vector space of dimension $n$ and $0 \le r \le n$ is another integer.

Let $W_0 \subset \Hom(V, V^*)$ be the space of skew-symmetric maps, let $W_1 \subset \Hom(V^*, V)$ be the space of symmetric maps, and put
  \[
    W=W_0 \times W_1,
  \]
  all thought of as affine varieties. We will generally refer to points of $W$ as pairs $(f,g)$. One can identify $W_0$ with $\lw^2(V^*)$ and $W_1$ with $\Sym^2(V)$, but we prefer to think in terms of linear maps.

Let $Z_0 \subseteq W_0$ be the (reduced) subvariety of maps of rank $\le 2(n-r)$, and let $Z_1 \subseteq W_1$ be the (reduced) subvariety of maps of rank $\le 2r$.  Set
  \[
    Z=Z_0 \times Z_1.
  \]
Note that $Z_0=W_0$ if $2r \le n$ and that $Z_1=W_1$ if $2r \ge n$.
  
\subsection{Invariant theory} \label{ss:invt-theory}

For this section, we assume that $0 < 2r \le n$. Our goal is to describe a certain $\GL(V)$-equivariant double covering $Z' \to Z$. In the case that $2r=n$, we will have $Z' = \Spec(\cO_Z[y]/(y^2 - \det g))$ where $g$ represents the generic symmetric matrix, but in all other cases, the construction is more subtle.

First, we describe an alternative construction of $Z$ (more specifically, $Z_1$). Let $E$ be a vector space of dimension $2r$ equipped with a nondegenerate symmetric bilinear form. Let $\bO(E)$ be the corresponding orthogonal group. Note that if $x \in \bO(E)$, then $\det x = \pm 1$; let $\SO(E)$ be the special orthogonal group, which is the index 2 subgroup of $\bO(E)$ consisting of matrices with determinant 1.

Now consider the space of linear maps $X = \hom(E,V)$. Let $J \colon E^* \to E$ be the isomorphism induced by the orthogonal form on $E$. We have a surjective $\bO(E)$-equivariant map
\[
  X \to Z_1, \qquad \phi \mapsto \phi J \phi^*,
\]
which identifies $\cO_{Z_1}$ with the $\bO(E)$-invariant subring of $\cO_X$ (this combines what is typically referred to as the first and second fundamental theorems of invariant theory for the orthogonal group, see for instance \cite[5.2.2, 12.2.14]{GW}). So $Z_1$ has rational singularities \cite[Corollaire]{boutot}.

We can also define an extended version of this map
\[
  X \to Z_1 \times \bigwedge^{2r} V
\]
where the first component is as before, and the second map records the Pl\"ucker coordinates of $\phi$, i.e., with respect to some basis, we are taking the maximal minors of $\phi$. Let $Z'_1$ be the image of this extended map and set
\begin{align} \label{eqn:Z'}
  Z' = Z_0 \times Z'_1 = W_0 \times Z'_1.
\end{align}
This map is $\SO(E)$-equivariant and identifies $\cO_{Z'_1}$ with the $\SO(E)$-invariant subring of $\cO_X$. \cite[Theorem 10.2]{KP}. In particular, $Z'_1$ and $Z'$ have rational singularities (and hence are normal and Cohen--Macaulay). When discussing $Z'$ or $Z'_1$, we will refer to the coordinates of $\bigwedge^{2r} V$ as Pl\"ucker coordinates.

Finally, $\Gamma = \bO(E)/\SO(E)$ acts on $Z'_1$, and $\cO_{Z_1}$ is the $\Gamma$-invariant subring of $\cO_{Z'_1}$. Since $\Gamma \cong \bZ/2$, we see that $Z'_1 \to Z_1$ is finite map of degree 2, and similarly so is $Z' \to Z$. 

So far, the discussion has been mostly about $Z_1$ and $Z'_1$, so now we bring in $W_0$. The Pl\"ucker coordinates on $Z_1$ span the space $\bigwedge^{2r}(V^*)$ while the span of the Pfaffians of order $2r$ on $W_0$ is the space $\bigwedge^{2r}V$. We will consider the 1-dimensional space of $\GL(V)$-invariants in $\bigwedge^{2r}V \otimes \bigwedge^{2r}(V^*)$.

To pin down a specific $\GL(V)$-invariant element $\Phi$, let's pick a basis $e_1,\dots,e_n$ for $V$ and a basis $x_1,\dots,x_{2r}$ for $E$ such that $J = \begin{pmatrix} 0 & I_r \\ I_r & 0 \end{pmatrix}$. Note that the choice of basis for $V$ will not matter because of the $\GL(V)$-equivariance; different choices of bases for $E$ will affect the definition of $\Phi$ by potentially a sign, so {\it we will fix this basis for $E$ for all future computations.}

Using these bases, we can write $\phi$ as an $n \times 2r$ matrix and $f \in W_0$ as an $n \times n$ skew-symmetric matrix (use the dual basis $e_1^*,\dots,e_n^*$ for $V^*$). Given a size $2r$ subset $S$ of $\{1,\dots,n\}$, let $\phi_S$ be the determinant of the submatrix of $\phi$ whose rows are indexed by $S$, and let $f_S$ be the Pfaffian of the submatrix of $f$ whose rows and columns are indexed by $S$. Then we define
\begin{align} \label{eqn:Phi}
  \Phi = \sum_S f_S \phi_S.
\end{align}

Finally, we will later need to compare this construction to the one in \cite[\S 6.3]{weyman}, so we explain that now using our current notation.

Consider the Grassmannian $\Gr_r(V)$ of rank $r$ subspaces of $V$ with tautological sequence
\[
  0 \to \cR \to V \otimes \cO_{\Gr_r(V)} \to \cQ \to 0
\]
where $\cR$ is the tautological rank $r$ subbundle. We have an inclusion $\Sym^2(\cQ^*) \subset \Sym^2(V^*) \otimes \cO_{\Gr_r(V)}$ and let $\eta$ denote the quotient. This is the construction considered in \cite[\S 6.3, ``The second incidence variety'']{weyman} (in addition to the permuted notation, we note that under the duality $\Gr_r(V) \cong \Gr_{n-r}(V^*)$, the roles of the tautological subbundle and quotient bundles are swapped). Then \cite[Proposition 6.3.3]{weyman} computes the Tor groups of $\rH^0(\Gr_r(V), \Sym \eta)$, regarded as a module over $\Sym(W_1^*)$. This is isomorphic to $Z'_1$ by \cite[Proposition 4.10]{lwood} (in the notation there, $\ol{B}_\emptyset = \cO_{Z_1}$ while $\ol{M}_\emptyset = \rH^0(\Gr_r(V), \Sym \eta)$).

\subsection{Statement of results} \label{ss:pedetvar}

In what follows, $T$ denotes a general $\bC$-algebra and $V_T = T \otimes V$. We consider the following situation 
\begin{itemize}
\item Let $\chi(u) \in \cO_Z[u]$ denote the characteristic polynomial of $fg$, where as above, a general point of $Z$ is referred to as a pair $(f,g)$.
  \begin{itemize}
  \item If $2r > n$, let $\ol{\chi}(u) = \chi(u)/u^{2r-n}$, define $Z'=Z$, and set
    \[
      \tilde{Z} = \Spec(\SFact^{n-r,0}_{\cO_Z}(\ol{\chi})), \qquad \tilde{\cZ} = \Spec(\SSplit_{\cO_Z}(\ol{\chi})).
    \]
    
  \item If $2r \le n$, let $\ol{\chi}(u) = \chi(u) / u^{n-2r}$, define $Z'$ as in \eqref{eqn:Z'}, and set
    \[
      \tilde{Z} = \Spec(\DFact^{r,0}_{\cO_{Z'}}(\ol{\chi}, \Phi)), \qquad \tilde{\cZ} = \Spec(\DSplit_{\cO_{Z'}}(\ol{\chi}, \Phi)),
    \]
    where $\Phi$ is as defined in \eqref{eqn:Phi}.
    We will show in Lemma~\ref{lem:chi-bar} that $\ol{\chi}$ is indeed a polynomial.
  \end{itemize}
  
\item Let $Y$ be the scheme defined as follows: a $T$-point is a tuple $(f, g, R)$ where:
\begin{itemize}
\item $R \subset V_T$ is a $T$-submodule that is locally a rank $r$ summand.
\item $f \colon V_T \to V^*_T$ is a skew-symmetric map of $T$-modules for which $R$ is isotropic.
\item $g \colon V_T^* \to V_T$ is a symmetric map of $T$-modules for which $(V_T/R)^*$ is isotropic.
\end{itemize}
This is a vector bundle over the Grassmannian $\Gr_r(V)$.
\item Let $\pi \colon Y \to \tilde{Z}$ be the map taking $(f,g,R)$ to $(f,g,p)$ where $p$ is the characteristic polynomial of $fg$ on $(V_T/R)^*$ if $2r > n$, and is the characteristic polynomial of $gf$ on $R$ if $2r\le n$. (We prove this is well-defined in Lemma~\ref{lem:neg-div}.)
  
\item If $2r > n$, let $\cY$ be the scheme defined as follows: a $T$-point is a tuple $(f, g, F_\bullet)$ where:
  \begin{itemize}
  \item $F_\bullet = (F_{2r-n} \subset F_{2r-n+1} \subset \cdots \subset F_{n-1})$ is a flag of locally free $T$-summands of $V_T$ where $\rank F_i = i$.
  \item $f \colon V_T \to V^*_T$ is a skew-symmetric map of $T$-modules such that $f(F_{r-i}) \subseteq (V_T/F_{r+i})^*$ for $i \ge 0$. Note that this forces $F_r$ to be isotropic for $f$.
\item $g \colon V_T^* \to V_T$ is a symmetric map of $T$-modules such that $g((V_T/F_{r+i})^*) \subseteq F_{r-i}$ for $i \ge 0$. Note that this forces $(V_T/F_r)^*$ to be isotropic for $g$.
\end{itemize}

This is a vector bundle over the partial flag variety $\Fl(2r-n, 2r-n+1,\dots, n-1, V)$.

Let $\rho \colon \cY \to \tilde{\cZ}$ be the map taking $(f,g,F_\bullet)$ to $(f,g,\lambda_1,\dots,\lambda_{n-r})$ where $\lambda_i$ is the eigenvalue of $fg$ on $(F_{n-i+1}/F_{n-i})^*$. (We prove this is well-defined in Lemma~\ref{lem:neg-div}.)

\item If $2r \le n$, let $\cY$ be the scheme defined as follows: a $T$-point is a tuple $(f, g, F_\bullet)$ where:
  \begin{itemize}
  \item $F_\bullet = (F_1 \subset F_{2} \subset \cdots \subset F_{2r})$ is a flag of locally free $T$-summands of $V_T$ where $\rank F_i = i$.
  \item $f \colon V_T \to V^*_T$ is a skew-symmetric map of $T$-modules such that $f(F_{r-i}) \subseteq (V_T/F_{r+i})^*$ for $i \ge 0$. Note that this forces $F_r$ to be isotropic for $f$.
\item $g \colon V_T^* \to V_T$ is a symmetric map of $T$-modules such that $g((V_T/F_{r+i})^*) \subseteq F_{r-i}$ for $i \ge 0$. Note that this forces $(V_T/F_r)^*$ to be isotropic for $g$.
\end{itemize}
This is a vector bundle over the partial flag variety $\Fl(1, 2, \dots, 2r, V)$.

Let $\rho \colon \cY \to \tilde{\cZ}$ be the map taking $(f,g,F_\bullet)$ to $(f,g,\lambda_1,\dots,\lambda_{r})$ where $\lambda_i$ is the eigenvalue of $gf$ on $F_i/F_{i-1}$. (We prove this is well-defined in Lemma~\ref{lem:neg-div}.)
\end{itemize}
The purpose of \S \ref{s:pedetvar} is to study $\tilde{Z}$. Our main result is the following theorem:

\begin{theorem} \label{thm:pedetvar}
  We have the following:
\begin{enumerate}
\item $\tilde{\cZ}$ and $\tilde{Z}$ are integral and have rational singularities (and are thus normal and Cohen--Macaulay).
\item The maps $\tilde{\cZ} \to Z'$ and $\tilde{Z} \to Z'$ are finite flat; in fact, $\cO_{\tilde{\cZ}}$ is a free $\cO_{Z'}$-module of rank $s_0$ and $\cO_{\tilde{Z}}$ is a free $\cO_{Z'}$-module of rank $s_1$ where
  \[
    s_0 = \begin{cases} 2^{n-r} (n-r)! & \text{if $2r > n$} \\ 2^{r-1} r! & \text{if $2r \le n$} \end{cases}, \qquad
    s_1 = \begin{cases} 2^{n-r} & \text{if $2r > n$} \\ 2^{r-1}& \text{if $2r \le n$} \end{cases}.
  \]

\item If $2r > n$, equip $\bC^{2n-2r}$ with a symplectic form. We have isomorphisms of graded rings
  \begin{align*}
    \cO_{\tilde{\cZ}} \otimes_{\cO_Z} \bC &\cong \rH^*_{\sing}(\bI\Fl(\bC^{2n-2r}), \bC),\\
    \cO_{\tilde{Z}} \otimes_{\cO_Z} \bC &\cong \rH^*_{\sing}(\IGr_{n-r}(\bC^{2n-2r}), \bC),
  \end{align*}
  where the $\bI\Fl$ and $\IGr$ are the variety of isotropic flags, and isotropic subspaces, respectively.

\item If $2r \le n$, equip $\bC^{2r}$ with an orthogonal form. We have isomorphisms of graded rings
  \begin{align*}
    \cO_{\tilde{\cZ}} \otimes_{\cO_Z} \bC &\cong \rH^*_{\sing}(\bO\Fl(\bC^n), \bC) \otimes ( \bC \oplus \bigwedge^{2r} V^*),\\
    \cO_{\tilde{Z}} \otimes_{\cO_Z} \bC &\cong \rH^*_{\sing}(\IGr_r(\bC^n), \bC) \otimes ( \bC \oplus \bigwedge^{2r} V^*),
  \end{align*}
  where the $\bO\Fl$ and $\bO\Gr$ are the variety of isotropic flags, and isotropic subspaces, respectively,   with the conventions as in \S\ref{sec:typeD}, $\bC \oplus \bigwedge^{2r} V^*$ is a graded ring with $\deg(\bC) = 0$ and $\deg(\bigwedge^{2r} V^*) = r$; $\GL(V)$ acts on $\bigwedge^{2r} V^*$ in the natural way and trivially on everything else.

\item The maps $\rho \colon \cY \to \tilde{\cZ}$ and $\pi \colon Y \to \tilde{Z}$ are proper and birational. Moreover, we have $\rho_*(\cO_{\cY})=\cO_{\tilde{\cZ}}$ and $\pi_*(\cO_Y) = \cO_{\tilde{Z}}$ and $\rR^i \rho_*(\cO_{\cY}) = \rR^i \pi_* (\cO_Y) = 0$ for $i>0$.
\end{enumerate}
\end{theorem}

\subsection{Normal forms} \label{ss:normal-form}

Here we discuss some linear algebra related to $Z$. Recall that $(f,g) \in Z$ means that $f \colon V \to V^*$ is a skew-symmetric matrix of rank $\le 2(n-r)$ and that $g\colon V^* \to V$ is a symmetric matrix of rank $\le 2r$. First we consider the case $2r > n$. 

\begin{proposition}
  Suppose that $2r > n$. 
  Given $(f,g) \in Z$, there exists a basis $e_1,\dots,e_n$ for $V$ such that in block matrix form (using the dual basis $e_1^*,\dots,e_n^*$ for $V^*$), with block sizes $2r-n$ and $2n-2r$, we have
  \[
    f = \begin{pmatrix} 0 & 0 \\ 0 & A \end{pmatrix}, \qquad
    g = \begin{pmatrix} B & C \\ C^T & D \end{pmatrix}
  \]
  where
  \begin{itemize}
  \item $A$ is a skew-symmetric $(2n-2r) \times (2n-2r)$ matrix which is lower-triangular with respect to the antidiagonal.
  \item $B$ is a symmetric $(2r-n) \times (2r-n)$ matrix, 
  \item $C$ is a $(2r-n) \times (2n-2r)$ matrix, and 
  \item $D$ is a symmetric $(2r-2n) \times (2r-2n)$ matrix which is upper-triangular with respect to the antidiagonal.
  \end{itemize}
\end{proposition}

\begin{proof}
  In the notation of the previous section, the result is equivalent to showing that the composition $\cY \xrightarrow{\rho} \tilde{\cZ} \to Z$ is surjective, where the second map is the structure map for the signed splitting ring. Since $\rho$ is proper and $Z$ is irreducible, it suffices to show that there is a nonempty dense subset $U$ of $Z$ such that every point of $U$ has a nonempty fiber. We will take $U$ to be the  locus of pairs $(f,g)$ such that $\rank f = 2n-2r$, $g$ is invertible, and the $2n-2r$ nonzero eigenvalues of $gf$ are distinct. Then given $(f,g) \in U$, we take $e_1,\dots, e_{2r-n}$ to be any basis for $\ker f$. Order the eigenvectors with nonzero eigenvalues of $gf$ as $e_{2r-n+1},\dots,e_n$ so that the eigenvalues appear as $\lambda_1,\lambda_2\dots,-\lambda_2,-\lambda_1$. With respect to this basis, $A$ and $D$ are actually antidiagonal, so we're done.
\end{proof}

With the notation above, we have
\[
  fg = \begin{pmatrix} 0 & 0 \\ AC^T & AD \end{pmatrix}
\]
where again the block sizes are $2r-n$ and $2n-2r$. Let $\chi_{fg}(u)$ be the characteristic polynomial of $fg$. Then we see that $u^{2r-n}$ divides $\chi_{fg}(u)$, so that $\ol{\chi}_{fg}(u) = \chi_{fg}(u) / u^{2r-n}$ is also a polynomial.

Since $A$ is lower-triangular with respect to the antidiagonal and $D$ is upper-triangular with respect to the antidiagonal, $AD$ is lower-triangular (with respect to the usual diagonal). Let $a_1,a_2,\dots,a_{n-r}, -a_{n-r} , \dots, -a_1$ be the antidiagonal entries of $A$ (reading from top row to bottom) and let $d_1, \dots, d_{n-r}, d_{n-r}, \dots, d_1$ be the antidiagonal entries of $D$. Then the diagonal entries of $fg$ are
\begin{align*} \label{eqn:roots}
  a_1d_1, \dots, a_{n-r} d_{n-r}, -a_{n-r} d_{n-r}, \dots, -a_1d_1.
\end{align*}
These are the roots of $\ol{\chi}_{fg}(u)$.

Finally, we consider the case $2r \le n$. The proof for the next result is essentially the same as before, so we omit it.

\begin{proposition}
  Suppose that $2r \le n$. 
  Given $(f,g) \in Z$, there exists a basis $e_1,\dots,e_n$ for $V$ such that in block matrix form (using the dual basis $e_1^*,\dots,e_n^*$ for $V^*$), with block sizes $2r$ and $n-2r$, we have
  \[
    f = \begin{pmatrix} A & B \\ -B^T & C \end{pmatrix}, \qquad
    g = \begin{pmatrix} D & 0 \\ 0 & 0 \end{pmatrix}
  \]
  where
  \begin{itemize}
  \item $A$ is a skew-symmetric $2r \times 2r$ matrix which is lower-triangular with respect to the antidiagonal, 
  \item $B$ is a $2r \times (n-2r)$ matrix,
  \item $C$ is a skew-symmetric $(n-2r) \times (n-2r)$ matrix,    
  \item $D$ is a symmetric $2r \times 2r$ matrix which is upper-triangular with respect to the antidiagonal.
  \end{itemize}
\end{proposition}

With the notation above, we have
\[
  gf = \begin{pmatrix} DA & DB \\ 0 & 0 \end{pmatrix}
\]
where again the block sizes are $2r$ and $n-2r$. Let $\chi_{gf}(u)$ be the characteristic polynomial of $gf$. Then we see that $u^{n-2r}$ divides $\chi_{gf}(u)$, so that $\ol{\chi}_{gf}(u) = \chi_{gf}(u) / u^{n-2r}$ is also a polynomial.

Since $D$ is upper-triangular with respect to the antidiagonal and $A$ is lower-triangular with respect to the antidiagonal, $DA$ is upper-triangular (with respect to the usual diagonal). Let $a_1,a_2,\dots,a_{r}, -a_{r} , \dots, -a_1$ be the antidiagonal entries of $A$ (reading from top row to bottom) and let $d_1, \dots, d_{r}, d_{r}, \dots, d_1$ be the antidiagonal entries of $D$. Then the diagonal entries of $fg$ are
\begin{align*} \label{eqn:roots}
  -a_1d_1, \dots, -a_{r} d_{r}, a_{r} d_{r}, \dots, a_1d_1.
\end{align*}
These are the roots of $\ol{\chi}_{gf}(u)$.

\begin{proposition} \label{prop:Phi-constant}
  With respect to the coordinates above, up to a sign, we have
  \[
    \Phi = (-a_1d_1)\cdots (-a_rd_r) = (-1)^r d_1\cdots d_r \Pf(A).
  \]
\end{proposition}

\begin{proof}
  Using the notation in \S\ref{ss:invt-theory}, we can use $e_1,\dots,e_n$ above as our choice of basis for $V$ and let $E$ be a $2r$-dimensional orthogonal space equipped with a basis so that the isomorphism $J \colon E^* \to E$ is the matrix $J = \begin{pmatrix} 0 & I_r \\ I_r & 0 \end{pmatrix}$.

  We write $D$ in block matrix form (with block sizes $r,r$) as $D=\begin{pmatrix} D_1 & D_2 \\ D_2^T & 0 \end{pmatrix}$. Define a block matrix $\psi = \begin{pmatrix} \frac12 D_1 & I_r \\ D_2^T & 0 \end{pmatrix}$; then $D = \psi J \psi^T$. In particular, if we define $\phi = \begin{pmatrix} \psi \\ 0 \end{pmatrix}$ where $0$ is the zero matrix of size $(n-2r) \times 2r$, then $\phi J \phi^T = g$.

  Finally, all minors of $\phi$ vanish except for the submatrix corresponding to $\psi$, so we have $\Phi = \Pf(A) \det(\psi)$. We finish by observing that $\Pf(A) = a_1\cdots a_r$ and $\det(\psi) = \det D_2 = d_1\cdots d_r$.
\end{proof}

\subsection{Proof of Theorem~\ref{thm:pedetvar}: part 1}

We use notation from \S \ref{ss:pedetvar}. We now prove the parts of Theorem~\ref{thm:pedetvar} that do not involve $\cY$ or $Y$.

It is well-known that $Z_0$ and $Z_1$ have rational singularities \cite[\S 6.4, discussion after (6.4.2)]{weyman} and \cite[\S 6.3]{weyman}, and in particular, they are normal and Cohen--Macaulay. It follows that $Z=Z_0 \times Z_1$ also has rational singularities. For the generic linear maps $f \colon V \otimes \cO_Z \to V^* \otimes \cO_Z$ and $g \colon V^* \otimes \cO_Z \to V \otimes \cO_Z$, let $\chi(u) \in \cO_Z[u]$ be the characteristic polynomial for $fg$.

\begin{lemma} \label{lem:chi-bar}
  If $2r > n$, then $\chi(u)$ is divisible by $u^{2r-n}$. If $2r \le n$, then $\chi(u)$ is divisible by $u^{n-2r}$. In both cases, the quotient $\ol{\chi}(u)$ is a polynomial in $u^2$. 
\end{lemma}

This result will be refined by Lemma~\ref{lem:neg-div}.

\begin{proof}
  This follows immediately from the discussion in \S\ref{ss:normal-form}.
\end{proof}

Now we consider Theorem~\ref{thm:pedetvar}(b,c,d).

If $2r > n$, then (b) follows from Propositions~\ref{prop:sign-split} and \ref{prop:sign-fact}, while (c) follows from Propositions~\ref{prop:split-flag} and \ref{prop:sign-fact-grass}.

Now suppose that $2r \le n$. We explained in \S\ref{ss:invt-theory} that $Z' \to Z$ is finite of degree 2. Also, (b) follows from Proposition~\ref{prop:d-sign-split} and \ref{prop:d-sign-fact}, while (d) follows from Propositions~\ref{prop:d-split-flag} and \ref{prop:d-sign-fact-grass}.

~

Let $\Delta$ be the discriminant of $\ol{\chi}$, thought of as an element of $\cO_{Z'}$.

\begin{proposition}
We have $\Delta\ne 0$ and $\ol{\chi}(0) \ne 0$, and so $\tilde{\cZ}$ and $\tilde{Z}$ are reduced.
\end{proposition}

\begin{proof}
  The first two claims follow immediately from the discussion in \S\ref{ss:normal-form}.  The claims about $\tilde{\cZ}$ and $\tilde{Z}$ being reduced follow from Propositions~\ref{prop:sign-split} and \ref{prop:sign-fact} when $2r > n$ and from Propositions~\ref{prop:d-sign-split} and \ref{prop:d-sign-fact} when $2r \le n$.
\end{proof}

We now discuss $V(\Delta)$, the locus of $(f,g)$ such that $\ol{\chi}$ has a repeated root. There are two ways that this can happen:
\begin{enumerate}[\indent (1)]
\item $\ol{\chi}_{fg}(0)=0$ (and then $0$ must appear as a root with multiplicity $\ge 2$ since $\ol{\chi}_{fg}(u)$ is a polynomial in $u^2$), or
\item there exist four roots of the form $\alpha,\alpha,-\alpha,-\alpha$.
\end{enumerate}

Both of these are closed conditions, and we let $V(\Delta)_0$ denote the subvariety of $V(\Delta)$ where (1) happens, and $V(\Delta)_A$ denote the subvariety of $V(\Delta)$ where (2) happens.

\begin{proposition}
  The variety $V(\Delta)_A$ is irreducible. In fact, any two points of $V(\Delta)_A$ can be joined by an irreducible rational curve contained in $V(\Delta)_A$.

  If $r \ne n/2$, then $V(\Delta)_0$ is also irreducible.

  If $r=n/2$, then $V(\Delta)_0$ has two irreducible components defined by the equations $\Pf(f)=0$ and $\det(g) = 0$.
\end{proposition}

\begin{proof}
  To prove the statement about $V(\Delta)_A$, we can combine the idea from \cite[Proposition 3.8]{superres} with the normal forms in \S\ref{ss:normal-form}.

  Now we consider $V(\Delta)_0$. First we suppose that $2r > n$. In that case, $Z$ has codimension $\binom{2r-n}{2}$ in $W$ and the locus where $\rank f < 2n-2r$ (and hence $\rank f \le 2n-2r-2$) has codimension $\binom{2r-n+2}{2}$ in $W$. So the codimension (in $Z$) of the locus where the rank of $f$ is not maximal is $4r-2n+1 \ge 3$. In particular, if we let $U$ denote the set of $(f,g)$ such that $\rank f = 2n-2r$, then $U \cap V(\Delta)$ is dense, so it suffices to show that $U \cap V(\Delta)_0$ and $U \cap V(\Delta)_A$ are irreducible. 

  Let $(f_1,g_1)$ and $(f_2,g_2)$ be two points in $U \cap V(\Delta)_0$. Fix any basis for $V$; then there exist change of bases $\gamma_i$ for $i=1,2$ such that
  \[
    \gamma_i \cdot f_i = \begin{pmatrix} 0 & 0 \\ 0 & A_i \end{pmatrix}, \qquad \gamma_i \cdot g_i = \begin{pmatrix} B_i & C_i \\ C_i^T & D_i \end{pmatrix}
  \]
  as in \S\ref{ss:normal-form}. Let $x_{i,1},\dots,x_{i,n-r},-x_{i,n-r},\dots,-x_{i,1}$ be the antidiagonal entries of $A_i$ (read top to bottom) and let $y_{i,1},\dots,y_{i,n-r},y_{i,n-r},\dots,y_{i,1}$ be the antidiagonal entries of $D_i$. Since $(f_1,g_1) \in V(\Delta)_0$, there exists $j$ such that $x_{1,j}y_{1,j}=0$. We can multiply $\gamma_1$ by a permutation matrix to assure that $x_{1,1}y_{1,1}=0$. Similarly, we may assume that $x_{2,1}y_{2,1}=0$. Since $\rank f_i = 2n-2r$, we see that $x_{i,1} \ne 0$, and hence $y_{1,1}=y_{2,1}=0$.

  Now we define matrices $A(t), B(t), C(t), D(t)$ with polynomial entries in $t$ by
  \begin{align*}
    A(t) &= tA_1 + (1-t) A_2, \qquad
    B(t) = tB_1 + (1-t) B_2, \\
    C(t) &= tC_1 + (1-t) C_2, \qquad
    D(t) = tD_1 + (1-t) D_2,
  \end{align*}
  and define $(f(t),g(t))$ by
  \[
    f(t) = \begin{pmatrix} 0 & 0 \\ 0 & A(t) \end{pmatrix}, \qquad  g(t) = \begin{pmatrix} B(t) & C(t) \\ C(t)^T & D(t) \end{pmatrix}.
  \]
  Then it is clear that $(f(t), g(t)) \in V(\Delta)_0$ for all $t$, and hence $\gamma_i \cdot (f_i,g_i)$ are in the same irreducible component for $i=1,2$. Next, $\GL(V)$ is a connected group, so its action preserves irreducible components. We conclude that $V(\Delta)_0$ is irreducible.

  The case when $2r < n$ can be handled in an analogous way (though the roles of $f$ and $g$ are reversed), but we will comment on the dimension count since it is a little different. In this case, $g$ generically has rank $2r$, and the codimension of $Z$ in $W$ is $\binom{n-2r+1}{2}$. The locus where $\rank g \le 2r-1$ has codimension $\binom{n-2r+2}{2}$ in $W$, and hence codimension $n-2r+1 \ge 2$ in $Z$. So again, this locus can be ignored for the purposes of determining the irreducible components of $V(\Delta)$.
  
  Finally, we consider $V(\Delta)_0$ when $2r=n$. In that case, $\chi(0)=\det(f) \det(g)$, and hence $\chi(0)=0$ means that $\det(f)=0$ or $\det(g)=0$. Then we're done since $\det(g)$ is an irreducible polynomial and $\det(f) = (\Pf f)^2$, where $\Pf$ is the Pfaffian, which is also irreducible (the irreducibility is well-known, but see \cite[\S\S 6.3, 6.4]{weyman} for details).
\end{proof}

\begin{proposition} \label{prop:tildeZ-normal}
  $\tilde{\cZ}$ and $\tilde{Z}$ are normal.
\end{proposition}

\begin{proof}
First suppose that $2r > n$. By Proposition~\ref{prop:sign-fact}(e), it suffices to show that $\tilde{\cZ}$ is normal. We verify the conditions of Proposition~\ref{prop:sign-norm-crit2}. We already know that $Z$ is normal and that $\Delta$ is a non-zerodivisor. It thus suffices to show that $V(\Delta,\partial \Delta)$ has codimension $\ge 2$. It suffices to show that the intersection of $V(\Delta, \partial \Delta)$ with each irreducible component of $V(\Delta)$ is a proper subset of that component. We do this by writing down a point in each component of $V(\Delta)$ that does not belong to $V(\Delta, \partial \Delta)$. We furthermore restrict to smooth points $x \in Z$ so that membership in $V(\Delta,\partial\Delta)$ means that $\Delta(x)=0$ and $d\Delta$ is 0 in the cotangent space of $x$.

First pick bases for $V_0$ and $V_1$. We will define a $\bC[\epsilon]/(\epsilon^2)$ point of $Z$ by
\begin{displaymath}
f=\begin{pmatrix} 0 & 0 \\ 0 & A \end{pmatrix}, \qquad
g=\begin{pmatrix} 0 & 0 \\ 0 & B \end{pmatrix}
\end{displaymath}
where $A$ and $B$ are square matrices of size $2n-2r$.
For the component $V(\Delta)_0$, we take
\begin{displaymath}
A = \begin{pmatrix}
0 & 1 \\
-1 & 0 \\
&& D \end{pmatrix}, \qquad
B = \begin{pmatrix}
  \epsilon & 0 \\
  0 & 1 \\
  && I
\end{pmatrix},
\end{displaymath}
where $D$ is a block diagonal matrix with $2 \times 2$ blocks $\begin{pmatrix} 0 & \lambda_i \\ -\lambda_i & 0 \end{pmatrix}$ with $\lambda_i$ distinct (up to sign) nonzero numbers not equal to $\pm 1$, and $I$ is an identity matrix of size $2n-2r-2$. At $\epsilon=0$, $f$ has maximal rank $2n-2r$, so we get a smooth point of $Z$. The value of $\Delta$ on $(f,g)$ is a nonzero scalar multiple of $\epsilon$, so this gives a tangent vector to a point of $V(\Delta)_0$ for which $d \Delta$ takes nonzero value.

Now we consider the component $V(\Delta)_A$. We instead define
\begin{displaymath}
A = \begin{pmatrix}
0 & 0 & 1 & 1 \\
0 & 0 & 1 & \epsilon \\
-\epsilon & -1 & 0 & 0 \\
-1 & -1 & 0 & 0 \\
&&&& D \end{pmatrix}, 
\end{displaymath}
where $D$ is a block diagonal matrix as before, but we require that $\lambda_i \notin \{-1,0,1\}$ and we take $B$ to be the identity matrix. Again, at $\epsilon=0$, $f$ has maximal rank $2n-2r$, so this gives a smooth point of $Z$. The characteristic polynomial of the upper-left $4 \times 4$ block of $A$ is $t^4 + (2+2\epsilon)t^2 + 1 - 2\epsilon$. Hence at $\epsilon=0$ it has roots $1,1,-1,-1$ and the corresponding $\bC$-point lies on $V(\Delta)_A$. Its discriminant is $16\epsilon$ (in general, the discriminant of a polynomial of the form $t^4 + bt^2 + c$ is $(b^2-4c)c$), and hence the discriminant of $\ol{\chi}_{fg}$ is a nonzero scalar multiple of $\epsilon$, so this point does not lie on $V(\Delta, \partial \Delta)$.

Now we consider the case $2r \le n$. We instead use Proposition~\ref{prop:d-sign-norm-crit2} and work with the reduced discriminant $\ol{\Delta}$. We need to instead work with $Z'$. We note that $Z' \to Z$ is \'etale on the locus where $\rank g = 2r$ (Proposition~\ref{prop:d-sign-split}), so the preimage of this locus in $Z'$ is smooth. We also note that $V(\Delta)=V(\ol{\Delta})$, and as above, it will suffice to show that $V(\ol{\Delta}, \partial \ol{\Delta})$ is a proper subset when intersected with any irreducible component of $V(\Delta)$.

The component $V(\Delta)_A$ can be handled exactly as above in the case $2r > n$ after we swap $r$ for $n-r$ (take either preimage under $Z' \to Z$). We note that at $\epsilon=0$, the point defined above satisfies $\rank g = 2r$, so defines a smooth point of $Z'$.

Next, consider the $\bC[\epsilon]/(\epsilon^2)$ point of $Z'$ (take either preimage under $Z' \to Z$) defined by
\[
  f = \begin{pmatrix}
    0_{n-2r} \\
 & 0 & \epsilon \\
 & -\epsilon & 0 \\
& && D \end{pmatrix}, \qquad g = \begin{pmatrix} 0_{n-2r} \\ & I_{2r} \end{pmatrix},
\]
where $0_{n-2r}$ is a zero matrix of size $(n-2r) \times (n-2r)$, $I_{2r}$ is an identity matrix of size $2r$, and $D$ is a block diagonal matrix with $2\times 2$ blocks $\begin{pmatrix} 0 & \lambda_i \\ -\lambda_i & 0 \end{pmatrix}$ with the $\lambda_1,\dots,\lambda_{r-1}$ distinct (up to sign) nonzero complex numbers. At $\epsilon=0$, $g$ has maximal rank $2r$, so defines a smooth point of $Z'$. Also, $\ol{\Delta}$ is a non-zero scalar multiple of $\Phi = \lambda_1 \cdots \lambda_{r-1} \epsilon$, so this point lies on $V(\Delta)_0 \setminus V(\ol{\Delta}, \partial \ol{\Delta})$

Hence we're done if $2r < n$. Finally suppose $2r=n$. We have $\cO_{Z'} = \Sym(W^*)[y] / (y^2 - \det g)$ and hence the singular locus of $Z'$ is the preimage of the singular locus of $V(\det g)$, i.e., the locus where $\rank g \le n-2$. In particular, points of $Z'$ are smooth if $\rank g \ge n-1$. The above point defined in the previous paragraph lies on $V(\Pf(f)) \setminus V(\ol{\Delta}, \partial \ol{\Delta})$.  Finally, we need to handle the component $V(\det g)$, in which case we will use the first example above (and take its preimage under $Z' \to Z$). This time, $\rank g = n-1$ at $\epsilon=0$, so is a smooth point of $Z'$ and hence defines a point of $V(\det g) \setminus V(\ol{\Delta}, \partial \ol{\Delta})$.
\end{proof}

\subsection{Proof of Theorem~\ref{thm:pedetvar}: part 2}

We now prove the parts of Theorem~\ref{thm:pedetvar} that relate to $\cY$, $Y$, $\rho$, and $\pi$. As before, we continue to use the notation from \S\ref{ss:pedetvar}.

By assumption, $(V/R)^*$ is an isotropic subspace for the symmetric form defined by $g$, which means that $g((V/R)^*) \subseteq R$. Similarly, $R$ is an isotropic subspace for the skew-symmetric form defined by $f$, so $f(R) \subseteq (V/R)^*$. In particular, $(V/R)^*$ is an invariant subspace for $fg$ and $R$ is an invariant subspace for $gf$.

\begin{lemma} \label{lem:neg-div}
  \begin{enumerate}
  \item If $2r > n$, let $p(u)$ be the characteristic polynomial of $fg$ on $(V/R)^*$.
  \item If $2r \le n$, let $p(u)$ be the characteristic polynomial of $gf$ on $R$. Then $p(0) = \Phi$ (up to a sign).
  \end{enumerate}
      In both cases, we have $p(u) p(-u) = \ol{\chi}(u)$.
\end{lemma}

\begin{proof}
  If $2r \le n$, the statement that $p(0) = \pm \Phi$ follows from Proposition~\ref{prop:Phi-constant}.

  The statement $p(u) p(-u) = \ol{\chi}(u)$ follows from the discussion in \S\ref{ss:normal-form}.
\end{proof}

The previous result shows that the maps $\pi \colon Y \to \tilde{Z}$ and $\rho \colon \cY \to \tilde{\cZ}$ are well-defined.
Let $E$ be the union of the following loci in $\tilde{\cZ}$:
  \begin{itemize}
  \item $\ol{\chi}(u)$ has a triple root, or 
  \item $\ol{\chi}(u)$ has two pairs of repeated roots, where the first pair is not the negative of the second, or
  \item $\ol{\chi}(u)$ has a unique pair of repeated roots $\lambda$ (up to sign), but the corresponding Jordan block of $fg$ is a scalar matrix (this is a rank condition on $fg - \lambda$ and hence is a closed condition). 
  \end{itemize}
  We set $U = \tilde{\cZ} \setminus E$.
  
\begin{proposition} \label{prop:pe-codim2}
 $\rho^{-1}(E)$ has codimension $\ge 2$ in $\cY$.  
\end{proposition}

\begin{proof}
  We will assume that $2r > n$; the other case can be handled in a similar way.
  
  First, $\cY$ is a vector bundle over the flag variety $\bF = \Fl(2r-n,\dots,n-1; V)$, and more specifically a subbundle of the trivial bundle $W \times \bF$. By equivariance, the restriction of $\rho^{-1}(E)$ to any fiber over $\bF$ is isomorphic to any other. So it suffices to show that within each fiber, these restrictions have codimension $\ge 2$.

To do concrete calculations, fix a flag $F_\bullet\in \bF$ and pick a basis $e_1,\dots,e_n$ adapted to this flag, i.e., so that $F_i = {\rm span}(e_1,\dots,e_i)$ for all $i=2r-n,\dots,n-1$ and let $e_1^*,\dots,e_n^*$ be the dual basis for $V^*$. Let $(f,g) \in W$ be a general point over the fiber of $F_\bullet$. With respect to our basis, and using the notation in \S\ref{ss:normal-form}, the roots of $\ol{\chi}(u)$ are the diagonal entries of $fg$, which can be written in the form 
\begin{align*} \label{eqn:roots}
  a_1d_1, \dots, a_{n-r} d_{n-r}, -a_{n-r} d_{n-r}, \dots, -a_1d_1.
\end{align*}
where the $a_i$ are certain coordinates of $f$ while the $d_i$ are certain coordinates of $g$.

Up to reordering indices, having a triple root means that we either have $a_1d_1=a_2d_2=a_3d_3$ or $a_1d_1=0=a_2d_2$, so this defines a codimension $\ge 2$ locus.

Similarly, up to reordering indices, having two pairs of repeated roots which are not negatives of each other has two cases: either we have $a_1d_1=a_2d_2$ and $a_3d_3=a_4d_4$ or we have $a_1d_1=a_2d_2$ and $a_3d_3=0$, which again defines a codimension $\ge 2$ locus.

Finally, suppose that there is a unique pair of repeated roots up to sign, say $\lambda$ appears twice. Having a repeated root gives a codimension 1 condition; generically $fg-\lambda$ has rank $n-1$, so having a scalar Jordan block imposes another independent condition that increases the codimension.
\end{proof}

\begin{lemma} \label{prop:pe-ratsing}
  The map $\rho^{-1}(U) \to U$ is an isomorphism; in particular, $U$ is smooth.
\end{lemma}

\begin{proof}
  Again, we are only going to prove this in the case $2r > n$ since the other case is similar.
  
Given a point of $\tilde{\cZ}$, if the roots of $\ol{\chi}(u)$ are distinct, then the signed splitting ring structure gives us an ordering of the eigenvalues $\lambda_1, \dots, \lambda_{n-r}, -\lambda_{n-r}, \dots, -\lambda_1$ (which are nonzero by definition of $U$) which forces its preimage under $\rho$ to be unique: for $i=1,\dots,n-r$, we must have $(V/F_{n-i})^*$ be the span of the $fg$-eigenspaces for $\lambda_1,\dots,\lambda_i$ and for $i=1,\dots,n-r-1$, $F_{2r-n+i}$ is the span of $\ker(gf)$ and the $gf$-eigenspaces for $-\lambda_1,\dots, -\lambda_i$. If instead $\ol{\chi}(u)$ has a repeated root (which is unique up to sign), then the span of the eigenspaces is still determined if the Jordan block is non-scalar (the eigenvector comes first and the generalized eigenvector gets used the next time).

In particular, we see that $\rho^{-1}(U) \to U$ is a bijection on $\bC$-points. Finally, $U$ is normal since it is an open subset of $\tilde{\cZ}$, which is normal by Proposition~\ref{prop:tildeZ-normal}, so this map is an isomorphism \stacks{02LR}.
\end{proof}

\begin{proposition}
The varieties $\tilde{\cZ}$ and $\tilde{Z}$ have rational singularities.
\end{proposition}

\begin{proof}
By Proposition~\ref{prop:pe-ratsing}, $U$ has rational singularities, and by Proposition~\ref{prop:pe-codim2}, $\rho^{-1}(\tilde{\cZ} \setminus U)$ has codimension $\ge 2$. Since $\tilde{\cZ}$ is normal by Proposition~\ref{prop:tildeZ-normal}, we can use \cite[Proposition~4.1]{superres} to conclude that $\tilde{\cZ}$ has rational singularities. Finally, $\tilde{Z}$ is a quotient of $\tilde{\cZ}$ by a finite group and hence inherits the rational singularities property by \cite[Proposition 5.13]{kollarmori}.
\end{proof}

\begin{corollary} \label{cor:affine-Z}
We have $\pi_*(\cO_Y)=\cO_{\tilde{Z}}$ and $\rR^i \pi_*(\cO_Y)=0$ for $i>0$.
\end{corollary}

\begin{proof}
This follows from the previous proposition and general results on rational singularities, see for example \cite[Proposition 4.2]{kollarmori}.
\end{proof}

\section{Cohomology of the periplectic Grassmannian} \label{sec:pegrass}

\subsection{Statement of results}

We fix the following notation for this section.
\begin{itemize}
\item $V$ is a complex vector space of dimension $n$.
\item $\bV = V \oplus V^*$ is a super vector space with $\deg(V)=0$ and $\deg(V^*)=1$. As before, we equip $\bV$ with the canonical symmetric pairing $\langle, \rangle$ given by
  \[
    \langle (a_1,b_1), (a_2,b_2) \rangle = b_1(a_2) + b_2(a_1).
  \]
\item We let $\bG = {\bf Pe}(\bV)$ be the stabilizer subgroup of supergroup $\GL(\bV)$ that preserves $\langle , \rangle$ (thought of as an element of $\Sym^2(\bV)$). Its Lie superalgebra $\mathfrak{pe}(\bV)$ has the $(\bZ/2)$-graded decomposition
  \[
    \mathfrak{pe}(\bV)_0 = \fgl(V),\qquad     \mathfrak{pe}(\bV)_1 = \Sym^2(V^*) \oplus \bigwedge^2(V).
  \]
  The even subgroup of $\bG$ is $\bG_{\rm ord} \cong \GL(V)$.
  
\item Let $0 \le r \le n$. We set $X = \bP\Gr_{r|n-r}(\bV)$ is the periplectic Grassmannian (see \S\ref{sec:GS-theory} for details).
\item $W = \bigwedge^2(V^*) \oplus \Sym^2(V)$ and $S = \Sym(W^*)$.
  
\item If $2r > n$, then set $A = \rH_{\rm sing}^*(\bI\Gr_{n-r}(\bC^{2(n-r)}), \bC)$ regarded as a graded $\bC$-algebra. 

  If $2r \le n$, then set $A = \rH_{\rm sing}^*(\bO\Gr_r(\bC^{2r}),\bC)$ regarded as a graded $\bC$-algebra.
\end{itemize}

\begin{theorem} \label{thm:grcoh}
Suppose $0 < r < n$.  We have the following:
\begin{enumerate}
\item We have a natural isomorphism $\rH^*(X, \cO_X)^{\bG}=A$ of graded algebras.
\item There is a canonical graded $\bG$-subrepresentation $E$ of $\rH^*(X, \cO_X)$ such that the natural map $A \otimes E \to \rH^*(X, \cO_X)$ is an isomorphism.
\item We have a canonical isomorphism of $\bG_\ord$-representations
\begin{displaymath}
E^i = \bigoplus_{p \ge 0} \Tor_p^S(\cO_Z, \bC)_{i+p}.
\end{displaymath}
\end{enumerate}
\end{theorem}

The restriction of $E$ to $\GL(E)$ will be described explicitly in \S\ref{sec:JPW}.

\begin{remark}
  As will follow from the discussion below, when $r=0$, the periplectic Grassmannian ${\bf PGr}_{0|n}(\bV)$ is topologically a point with coordinate ring $\bigwedge^\bullet(\Sym^2(V^*))$. Similarly, when $r=n$, it is topologically a point with coordinate ring $\bigwedge^\bullet(\bigwedge^2 V)$. 
\end{remark}

\subsection{Grothendieck--Springer theory}\label{sec:GS-theory}

Given a complex supercommutative superalgebra $T$, define $\bV_T = \bV \otimes_\bC T$, which inherits a $\bZ/2$-grading as well as the form $\langle, \rangle$.

The super Grassmannian $\Gr_{r|n-r}(\bV)$ represents the functor that assigns to $T$ the set of $T$-submodules of $\bV_T$ which are locally summands of rank $r|n-r$. Its underlying even scheme is $\Gr_r(V) \times \Gr_{n-r}(V^*)$.

The periplectic Grassmannian $X = \bP\Gr_{r|n-r}(\bV)$ is the closed subsuperscheme of $\Gr_{r|n-r}(\bV)$ consisting of subspaces which are isotropic with respect to $\langle, \rangle$. Let $X_{\rm ord}$ be its underlying even subscheme; we have an isomorphism $\Gr_r(V) \to X_{\rm ord}$ via $W \mapsto (W,(V/W)^*)$. Consider the restriction of the tautological sequence from $\Gr_{r|n-r}(\bV)$ to $X$:
\[
  0 \to \cR \to \cO_X \otimes \bV \to \cQ \to 0.
\]
Here $\cR$ is locally free of rank $r|n-r$ and $\cQ$ is locally free of rank $n-r|r$, and by the above comments, we see that $\cR_0 \cong \cQ_1^*$ and $\cQ_0 \cong \cR_1^*$.

We note that $\bG$ acts transitively on $X$. Pick a basis $e_1,\dots,e_n$ for $V$ with dual basis $e_1^*,\dots,e_n^*$ for $V^*$. Let $\bP_r$ be the stabilizer of the subspace with basis $e_1,\dots,e_r, e_{r+1}^*,\dots, e_n^*$. Then we can identify $X$ with the quotient $\bG/\bP_r$ (see \cite{masuoka} for background on quotients in super algebraic geometry).

Let $\cJ$ be the ideal sheaf of $X_{\rm ord}$ in $\cO_X$; we now determine the associated graded sheaf $\gr(\cO_X)$ with respect to the $\cJ$-adic filtration.

\begin{proposition}
  We have a natural isomorphism
  \[
    \bigwedge^2(\cR_0) \oplus \Sym^2(\cQ_0^*) \cong \cJ/\cJ^2
  \]
  of coherent $\cO_{X_{\rm ord}}$-modules.
\end{proposition}

\begin{proof}
  Since both sides are homogeneous bundles, it suffices to show that their restrictions over the point $x$, represented by the span of $e_1,\dots,e_r$, are isomorphic as $\bP_{\rm ord}$-modules. The right hand side is the odd component of the cotangent space, so via the identification $X = \bG/\bP$, we can identify it with $(T_e \bG)_1^* / (T_e \bP)^*_1$ where $e \in \bP$ is the identity. (See \cite[Proposition 4.18]{masuoka}.) We can identify $T_e \bG$ and $T_e \bP$ with the corresponding Lie algebras, and a routine calculation identifies the quotient with $\bigwedge^2R \oplus \Sym^2(V/R)^*$, where $R = {\rm span}(e_1,\dots,e_r)$.
\end{proof}

In particular, let $\epsilon = W^* \otimes \cO_{X_{\ord}}$; we have $\cJ / \cJ^2 \subseteq \epsilon$ and we define $\eta$ to be the quotient sheaf. In particular, we have
\[
  0 \to \cJ/\cJ^2 \to \epsilon \to \eta \to 0
\]
and are in the setup of \cite[\S 2.3]{superres}.
An element of the total space of $\epsilon$ can be represented by a tuple $(f,g,R)$ where $R \in \Gr_r(V)$; $f \colon V \to V^*$ is a skew-symmetric linear map, and $g \colon V^* \to V$ is a symmetric linear map. The total space of $\eta$ consists of tuples where both of the compositions
\[
  R \to V \xrightarrow{f} V^* \to R^*, \qquad (V/R)^* \to V^* \xrightarrow{g} V \to V/R
\]
are 0. In particular, the total space of $\eta$ is $Y$, as defined in \S\ref{ss:pedetvar}. 

By Corollary~\ref{cor:affine-Z}, $\tilde{Z}$ as defined in \S\ref{ss:pedetvar} is the affinization of $Y$, and $\pi \colon Y \to \tilde{Z}$ is the affinization map. In particular, by \cite[Theorem 2.4]{superres}, we have
\[
  \Tor_p^S(\cO_{\tilde{Z}}, \bC)_{p+q} = \rH^q(\Gr(r,V), \bigwedge^{p+q}( \bigwedge^2\cR_0 \oplus \Sym^2\cQ_0^*) ) ,
\]
and a spectral sequence
\begin{equation} \label{eqn:main-SS}
  \rE_1^{p,q} = \Tor^S_{-q}(\cO_{\tilde{Z}}, \bC)_p \Longrightarrow \rH^{p+q}(X, \cO_{X}).
\end{equation}

\subsection{The Jozefiak--Pragacz--Weyman complex} \label{sec:JPW}

Now define
\[
  \tilde{L}_k = \bigoplus_{p \ge 0} \Tor^S_p(\cO_{\tilde{Z}}, \bC)_{p+k}, \qquad
  L_k = \bigoplus_{p \ge 0} \Tor^S_p(\cO_{Z'}, \bC)_{p+k}
\]
(where, as usual, we set $Z'=Z$ if $2r>n$).

We now review the calculation of $L_k$ from \cite[\S\S 6.3, 6.4]{weyman}. We need to separate based on whether or not $2r>n$, though the answers are closely related in both cases.

Let $a,b$ be a nonnegative integers. Given a partition $\alpha$ with $\ell(\alpha) \le b$, let $\alpha^T$ denote the transpose partition, and define the partition
\begin{align*}
P(a, b,\alpha) = (b + \alpha_1, \dots, b + \alpha_b, \underbrace{b, \dots b}_{a}, \alpha^T_1,
\dots, \alpha^T_{\alpha_1}),
\end{align*}
whose Young diagram can be visualized as follows:
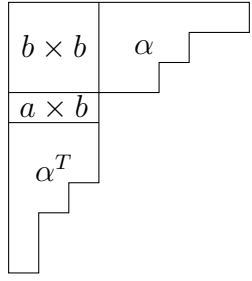
\[
\begin{tikzpicture}[scale=.4]
\draw (0,-2) -- (0,7) -- (8,7) -- (8,6) -- (6,6) -- (6,5) -- (5,5) -- (5,4) -- (3,4) -- (3,1) -- (2,1) -- (2,0) -- (1,0) -- (1,-2) -- (0,-2) -- cycle;
\draw (0,3) -- (3,3);
\draw (0,4) -- (3,4) -- (3,7);
\path (1.5,5.5) node {$b \times b$};
\path (1.5,3.5) node {$a \times b$};
\path (4.5,5.5) node {$\alpha$};
\path (1.5,1.5) node {$\alpha^T$};
\end{tikzpicture}
\]

Let $\bS_\lambda$ denote the Schur functor with highest weight $\lambda$.

\begin{proposition} \label{prop:Lk1}
  Suppose $2r > n$ and set $a = 2n-2r+1$. If $r<n$, then we have 
  \[
  L_{b(n-r)} = \bigoplus_\alpha \bS_{P(a,b,\alpha)} V
\]
and $L_k=0$ if $k$ is not divisible by $n-r$.

If $n=r$, then  $L_0 = \bigwedge^\bullet(\bigwedge^2 V)$ and all other $L_k$ are $0$.
\end{proposition}

\begin{proof}
If $n=r$, then $Z_0$ is a single point, so the minimal free resolution of $Z$ is the Koszul complex on $\bigwedge^2 V$.
  
In \cite[Proposition 6.4.3]{weyman}, the Tor groups are calculated for $Z_0$ in $\Sym(W_0^*)$, but we have
\[
  \Tor_i^S(\cO_Z, \bC) \cong \Tor_i^{\Sym(W_0^*)}(\cO_{Z_0}, \bC)
\]
since $\cO_Z = \cO_{Z_0} \otimes \Sym(W_1^*)$, so there is no difference. Specifically, we have
\[
  \Tor^S_i(\cO_Z, \bC) = \bigoplus_{\substack{b \ge 0,\ \alpha\\ i = |\alpha| + b(b+1)/2}} \bS_{P(a, b,\alpha)} V.
\]
The grading on $\Tor$ is determined by the size of the partition divided by 2 and we have
\[
  \frac12 |P(2n-2r+1, b,\alpha)|= |\alpha| + b(n-r) + \frac{b(b+1)}{2}. \qedhere
\]
\end{proof}

\begin{proposition} \label{prop:Lk2}
  If $0 < 2r \le n$, set $a = 2r-1$. Then we have
\[
  L_{br} = \bigoplus_\alpha \bS_{P(a,b,\alpha)} (V^*)
\]
and $L_k=0$ if $k$ is not divisible by $r$.

If $r=0$, then $L_0 = \bigwedge^\bullet(\Sym^2 (V^*))$ and all other $L_k$ are $0$.
\end{proposition}

\begin{proof}
  If $r=0$, then $Z'_1$ is a single point, so the minimal free resolution of $Z'$ is the Koszul complex on $\Sym^2(V^*)$.  

  As discussed in \S\ref{ss:invt-theory}, the Tor groups of $Z'_1$ as a module over $\Sym(W_1^*)$ are calculated in \cite[Proposition 6.3.3]{weyman}. But note that
\[
  \Tor_i^S(\cO_{Z'}, \bC) \cong \Tor_i^{\Sym(W_1^*)}(\cO_{Z_1'}, \bC)
\]
since $\cO_{Z'}= \Sym(W_0^*) \otimes \cO_{Z'_1}$, so there is no difference. Then
\[
  \Tor^S_i(\cO_{Z'}, \bC) = \bigoplus_{\substack{b \ge 0,\ \alpha\\ i = |\alpha| + b(b-1)/2}} \bS_{P(a, b,\alpha)}(V^*).
\]
The grading on $\Tor$ is determined by the size of the partition divided by 2 and we have
\[
  \frac12 |P(2r-1, b,\alpha)|= |\alpha| + br + \frac{b(b-1)}{2}. \qedhere
\]
\end{proof}

\begin{corollary} \label{cor:multfree}
  The $\bG_{\ord}$-representation $\bigoplus_{p \ge 0} \Tor_p^S(\cO_{Z'}, \bC)$ is multiplicity-free.
\end{corollary}

\subsection{Proof of Theorem~\ref{thm:grcoh}}

Now we assume $0<r<n$. If $2r > n$, we continue to use $Z'$ to mean $Z$.

\begin{proposition} \label{prop:tor-A}
  The graded vector space $\Tor_p^S(\cO_{\tilde{Z}}, \bC)$ is naturally a graded $A$-module, and the induced map
  \[
    A \otimes_\bC \Tor_p^S(\cO_{Z'}, \bC) \to \Tor_p^S(\cO_{\tilde{Z}}, \bC)
  \]
  is an isomorphism of graded $A$-modules.
\end{proposition}

\begin{proof}
  The proof is essentially the same as \cite[Proposition 6.8]{superres}.
\end{proof}

\begin{proposition} \label{prop:nocommonfactors}
  The $\bG_{\rm ord}$-representations $\tilde{L}_k$ and $\tilde{L}_{k+1}$ have no simple factors in common.
\end{proposition}

\begin{proof}
  Proposition~\ref{prop:tor-A} shows that $\tilde{L}_k = \bigoplus_{i \ge 0} A_{k-i} \otimes_\bC L_i$. Since $A$ is concentrated in even degrees, we see that $\tilde{L}_k$ is a sum of $L_i$'s with $i$ of the same parity as $k$. Hence the claim follows from the fact that each $L_k$ is multiplicity-free as a $\bG_{\rm ord}$-representation (Corollary~\ref{cor:multfree}).
\end{proof}

\begin{corollary}\label{cor:ss-degen}
  The spectral sequence \eqref{eqn:main-SS} degenerates at the $\rE_1$ page.
\end{corollary}

\begin{proof}
The spectral sequence is $\bG_{\rm ord}$-equivariant, so all of the differentials are 0 by Proposition~\ref{prop:nocommonfactors}.
\end{proof}

\begin{corollary}
We have canonical $\bG_\ord$-equivariant isomorphisms
\begin{displaymath}
\rH^i(X, \cO_X)=\gr(\rH^i(X, \cO_X))=\tilde{L}_i.
\end{displaymath}
\end{corollary}

\begin{proof}
Corollary~\ref{cor:ss-degen} gives a canonical isomorphism $\gr(\rH^i(X, \cO_X))=\tilde{L}_i$. It follows that $\gr(\rH^i(X, \cO_X))$ is multiplicity free as a representation of $\bG_\ord$, and so the same is true of $\rH^i(X, \cO_X)$. Thus the filtration on $\rH^i(X, \cO_X)$ canonically splits, which yields a canonical isomorphism $\rH^i(X, \cO_X)=\gr(\rH^i(X, \cO_X))$.
\end{proof}

Define $E=\rH^i(X, \cO_X)^{\bG_\ord}$.

\begin{proposition} \label{prop:G=G0-invt}
We have $\rH^i(X, \cO_X)^{\bG}=\rH^i(X, \cO_X)^{\bG_\ord}$.
\end{proposition}

\begin{proof}
  If $r=n$, then $X$ is a topologically a point with coordinate ring $\bigwedge^\bullet(\bigwedge^2 V)$ and so $\rH^0(X, \cO_X) = \bigwedge^\bullet(\bigwedge^2 V)$ and higher cohomology vanishes. Similarly, if $r=0$, then $\rH^0(X, \cO_X) = \bigwedge^\bullet(\Sym^2(V^*))$ and higher cohomology vanishes. The result is clear in both of these cases.

  So we assume that $0 < r < n$. The action of the odd piece of the Lie algebra of $\bG$ on $W$ is a $\bG_\ord$-equivariant map
\begin{displaymath}
(\Sym^2(V^*) \oplus \bigwedge^2 V) \otimes E \to \rH^i(X, \cO_X).
\end{displaymath}
It follows from Propositions~\ref{prop:Lk1} and \ref{prop:Lk2} that $\rH^i(X, \cO_X)$ does not contain a $\bG_\ord$-subrepresentation isomorphic to $\Sym^2 (V^*)$. Similarly, if $\rH^i(X, \cO_X)$ contains a $\bG_\ord$-subrepresentation isomorphic to $\bigwedge^2 V$ then we must be in the case $2r > n$ (if $2r \le n$, all of the representations are Schur functors of the dual $V^*$), and more specifically, this can only happen if $a=b=1$; i.e., $n=r$, which we have already discussed.

It follows that this map must be~0. Thus $E$ is annihilated by the Lie algebra of $\bG$, and it follows that $\bG$ acts trivially on $E$.
\end{proof}

Given a graded vector space $U$, we define the \defn{trivial filtration} by ${\rm Fil}^i(U)=\bigoplus_{j \ge i} U_j$. With respect to this filtration, we have a natural isomorphism $U=\gr(U)$.

\begin{proposition}
  We have a natural isomorphism $\rH^*(X, \cO_X)^{\bG}=A$ of graded algebras. Moreover, the filtration on $\rH^*(X, \cO_X)$ induces the trivial filtration on $A$.

  The natural map $A \otimes E \to \rH^*(X, \cO_X)$ is an isomorphism.
\end{proposition}

\begin{proof}
  The proof is similar to \cite[Propositions 6.13, 6.14]{superres}.
\end{proof}

\begin{proposition} \label{prop:Gext}
  Suppose that $M$ and $M'$ are representations of $\bG$ such that $M \cong L_i$ and $M' \cong L_j$ as $\bG_\ord$-representations, with $i \ne j$. Then $\Ext^1_{\bG}(M, M')=0$.

  $E$ is a $\bG$-subrepresentation of $\rH^*(X, \cO_X)$.
\end{proposition}

\begin{proof}
  The proof is similar to the proof of \cite[Propositions 6.15, 6.16]{superres}.
\end{proof}


\begin{thebibliography}{SSW}

\bibitem[AW1]{akin-weyman} Kaan Akin, Jerzy Weyman. Minimal free resolutions of determinantal ideals and irreducible representations of the Lie superalgebra $\gl(m \vert n)$. {\it J.\ Algebra} {\bf 197} (1997), no.~2, 559--583.

\bibitem[AW2]{akin-weyman3} Kaan Akin, Jerzy Weyman. Primary ideals associated to the linear strands of Lascoux's resolution and syzygies of the corresponding irreducible representations of the Lie superalgebra $\gl(m|n)$. {\it J.\ Algebra} {\bf 310} (2007), no.~2, 461--490.

\bibitem[Bo]{boutot} Jean-Fran{\c c}ois Boutot. Singularit\'es rationelles et quotients par les groupes r\'eductifs. {\it Invent. Math.} {\bf 88} (1987), 65--68.

\bibitem[Co]{coulembier} Kevin Coulembier. Bott--Borel--Weil theory, BGG reciprocity and twisting functors for Lie superalgebras. \textit{Transform.\ Groups} \textbf{21} (2016), no.~3, 681--723.
 \arxiv{1404.1416v4}

\bibitem[EG]{edidin-graham} Dan Edidin, William Graham. Characteristic classes and quadric bundles. {\it Duke Math. J.} {\bf 78} (1995), 277--299. \arxiv{alg-geom/9412007v1}

\bibitem[GW]{GW} Roe Goodman, Nolan R. Wallach. {\it Symmetry, Representations, and Invariants}. Graduate Texts in Math. 255, Springer, 2009.
  
\bibitem[GSS]{GSS} Daniel R. Grayson, Alexandra Seceleanu, Michael E. Stillman. Computations in intersection rings of flag bundles. \arxiv{1205.4190v1}
  
\bibitem[GS]{gruson} Caroline Gruson, Vera Serganova. Cohomology of generalized supergrassmannians and character formulae for basic classical Lie superalgebras. {\it Proc. Lond. Math. Soc.} {\bf 101} (2010), no.~3, 852--892. \arxiv{0906.0918v2}
  
\bibitem[KM]{kollarmori} J\'anos Koll\'ar, Shigefumi Mori. {\it Birational geometry of algebraic varieties}. Cambridge Tracts in Math. {\bf 134}, Cambridge Univ. Press, Cambridge, 1998.

\bibitem[KP]{KP} Hanspeter Kraft, Claudio Procesi. {\it Classical Invariant Theory. A Primer}. available from \url{https://dmi.unibas.ch/fileadmin/user_upload/dmi/Personen/Kraft_Hanspeter/Classical_Invariant_Theory.pdf}
  
\bibitem[Ls]{lascoux} Alain Lascoux. Syzygies des vari\'et\'es d\'eterminantales. {\it Adv.\ in Math.} {\bf 30} (1978), no.~3, 202--237.

\bibitem[MT]{masuoka} Akira Masuoka, Yuta Takahashi. Geometric construction of quotients $G/H$ in supersymmetry. {\it Transform. Groups} {\bf 26} (2021), 347--375. \arxiv{1808.05753v4}
  
\bibitem[Pe]{penkov} I. Penkov. Borel--Weil--Bott theory for classical Lie superalgebras. \textit{J.\ Soviet Math.} \textbf{51} (1990), 2108--2140.

\bibitem[PS]{penkov-serganova} I. Penkov, V. Serganova. Characters of irreducible $G$-modules and cohomology of $G/P$ for the Lie supergroup $G = Q(N)$. \textit{J.\ Math.\ Sci.} \textbf{84} (1997), no.~5, 1382--1412.

\bibitem[PW]{pragacz-weyman} Piotr Pragacz, Jerzy Weyman. Complexes associated with trace and evaluation. Another approach to Lascoux's resolution. {\it Adv.\ in Math.} {\bf 57} (1985), no.~2, 163--207.

\bibitem[RW]{raicu-weyman} Claudiu Raicu, Jerzy Weyman. Syzygies of determinantal thickenings and representations of the general linear Lie superalgebra. {\it Acta Math. Vietnam.} {\bf 44} (2019), no.~1, 269--284. \arxiv{1808.05649v1}.

\bibitem[Sa1]{sam} Steven V Sam. Derived supersymmetries of determinantal varieties. {\it J.\ Commut.\ Algebra} {\bf 6} (2014), no.~2, 261--286. \arxiv{1207.3309v1}

\bibitem[Sa2]{sam-osp} Steven V Sam. Orthosymplectic Lie superalgebras, Koszul duality, and a complete intersection analogue of the Eagon--Northcott complex. {\it J. Eur. Math. Soc (JEMS)} {\bf 18} (2016), no.~12, 2691--2732. \arxiv{1312.2255v2}

\bibitem[Sa3]{bwfact} Steven V Sam. Borel--Weil factorization for super Grassmannians. \arxiv{2407.05167v1}

\bibitem[Sa4]{bwfact2} Steven V Sam. Borel--Weil factorization for orthosymplectic Grassmannians. In progress.
  
\bibitem[SS]{superres} Steven V Sam, Andrew Snowden. Cohomology of flag supervarieties and resolutions of determinantal ideals. {\it Algebr. Geom.} {\bf 11} (2024), no.~1, 37--70. \arxiv{2108.00504v2}
  
\bibitem[SSW]{lwood} Steven V Sam, Andrew Snowden, Jerzy Weyman. Homology of Littlewood complexes, {\it Selecta Math. (N.S.)} {\bf 19} (2013), no.~3, 655--698. \arxiv{1209.3509v2}

\bibitem[SP]{stacks-project} The Stacks project authors, The Stacks project, \url{https://stacks.math.columbia.edu} (2021)
  
\bibitem[We]{weyman} Jerzy Weyman. {\it Cohomology of vector bundles and syzygies}. Cambridge Tracts in Math. {\bf 149}, Cambridge University Press, Cambridge, 2003.

\end{thebibliography}
\end{document}